\documentclass[10pt]{article}
\usepackage[utf8]{inputenc}
\usepackage[T1]{fontenc}

\usepackage[margin=2.5cm]{geometry}
\usepackage{multicol}
\usepackage{latexsym,mathtools,tikz,microtype,comment}

\usepackage{enumitem,url, hyperref}
\usepackage{tikz-cd}
\usepackage{mathrsfs}
\usepackage{cancel}

\usepackage{amsfonts,amsmath,amssymb,amsthm}

\newcommand{\PP}{\mathbb P}
\newcommand{\RR}{\mathbb R}
\newcommand{\ZZ}{\mathbb Z}
\newcommand{\CC}{\mathbb C}

\newcommand{\dual}[1]{#1^{\perp_h}}

\newtheorem{theorem}{Theorem}[section]
\newtheorem{lemma}[theorem]{Lemma}
\newtheorem{proposition}[theorem]{Proposition}
\newtheorem{corollary}[theorem]{Corollary}

\newtheorem{question}[theorem]{Question}

\theoremstyle{definition}
 
\newtheorem{remark}[theorem]{Remark}

\title{Kochen--Specker Configurations from Grids on Dual Quadrics}
\author{Giuseppe Favacchio}
\date{}

\begin{document}
\maketitle
\begin{center}
    \normalsize Dipartimento di Ingegneria, Universit\`a degli Studi di Palermo\\
Viale delle Scienze, 90128 Palermo, Italy\\
\texttt{giuseppe.favacchio@unipa.it}
\end{center}

\begin{abstract}
We develop a geometric framework for constructing and organizing Kochen--Specker configurations in real four-dimensional space. The construction uses finite grids on pairs of smooth quadrics related by Euclidean polarity. Their incidence geometry directly produces orthogonal measurement contexts and parity proofs of quantum contextuality, yielding infinite families of configurations and a geometric interpretation and extension of a previously known cyclic construction. For a distinguished subfamily, the same geometry admits a canonical completion determined by secant lines. The first two instances of this completion recover the exceptional root configurations of types $F_4$ and $H_4$, while the smallest case also recovers the Cabello configuration and its embedding in the Peres configuration. Thus several prominent four-dimensional contextual configurations, previously obtained from different constructions, arise from a single projective-geometric mechanism.
\end{abstract}

\thanks{\emph{keywords}:	Kochen--Specker configurations, contextuality, grids, quadrics, polarity, finite point configurations, geproci sets}


\section{Introduction}
\label{sec:introduction}

The Kochen--Specker theorem is one of the fundamental obstructions to a
noncontextual hidden-variable description of quantum mechanics
\cite{KochenSpecker1967}. In its finite geometric formulation, one
considers a collection of rays in a Hilbert space together with
distinguished orthogonal bases, or contexts, and asks whether it is
possible to assign $0$--$1$ values to the rays, independently of the
context in which they occur, with exactly one ray assigned the value
$1$ in each context. Finite configurations for which no such
assignment exists provide concrete realizations of Kochen--Specker
contextuality,  \cite{AbramskyBrandenburger2011,AcinEtAl2015,CabelloSeveriniWinter2014}. Beyond its foundational role, contextuality has become
a structural resource in quantum information \cite{BermejoVegaEtAl2017,Amaral2019,HowardEtAl2014}, making systematic
constructions of finite contextual measurement scenarios of
independent interest. We refer to \cite{BudroniEtAl2022} for a modern
account of the subject and its connections with quantum information,
graph theory, and experiments.

Kochen--Specker configurations are usually described through lists of
vectors, orthogonality graphs or hypergraphs, or combinatorial
constructions designed to produce the required contexts. Our point of
departure is that several four-dimensional examples carry a
substantial projective geometry which is not visible from the
orthogonality hypergraph alone. The purpose of this paper is to make
this geometry systematic. The basic objects are finite grids on smooth
quadrics in $\PP^3(\mathbb R)$, while the Euclidean inner product
supplies a polarity which converts part of their projective incidence
structure into orthogonality. Thus the construction separates two
geometric mechanisms: projective incidence organizes the
configurations, while Euclidean polarity turns selected incidences
into compatible measurement contexts.

Dimension four contains several of the most prominent finite examples.
Peres constructed a highly symmetric configuration of $24$ rays
\cite{Peres1991}, while Cabello, Estebaranz and Garc\'ia-Alcaine later
exhibited an $18$-ray, $9$-context parity proof
\cite{CabelloEstebaranzGarciaAlcaine1996}. The number of rays is
minimal: the Peres conjecture that every Kochen--Specker set contains
at least $18$ rays was proved in \cite{XuChenGuhne2020}. The relation
between the Cabello and Peres configurations is itself highly
structured: the Peres system contains a large family of parity proofs,
including configurations of type $18$--$9$
\cite{WaegellAravind2011Peres}. Further highly symmetric examples arise
from finite root systems and regular polytopes, notably from the
$600$-cell and the root system $H_4$
\cite{WaegellAravindMegillPavicic2011}. From a different direction,
the projectivized $F_4$ and $H_4$ root configurations are also notable
examples of geproci sets, finite subsets of $\PP^3$ whose general
projection to $\PP^2$ is a complete intersection
\cite{ChiantiniEtAl2022}. These examples suggest that the projective
geometry of four-dimensional contextual configurations deserves to be
studied in its own right.

Our construction starts from a smooth split quadric
$Q\subset\PP^3(\mathbb R)$ and its dual $\dual{Q}$ with respect to the
Euclidean polarity. Since a split quadric carries two rulings by
projective lines, finite choices of lines in the two rulings determine
finite grids. We show that two suitably positioned grids on the pair
$Q,\dual{Q}$ give rise to orthogonal contexts through pairs of polar
secants. In this way the relevant orthogonality is organized by
projective incidence and duality rather than introduced only at the
level of the orthogonality hypergraph.

Our first main result is an infinite geometric family of Kochen--Specker
configurations. For every pair of odd integers $p,q\geq3$, we construct
two $(p,q)$-grids
\[
A\subset Q,
\quad
B\subset\dual{Q},
\quad
\text{whose union}
\quad
Z_{p,q}=A\sqcup B
\]
consists of $2pq$ rays and supports $pq$ distinguished orthogonal
contexts. Every ray occurs in exactly two of these contexts, and since
$pq$ is odd, the usual parity argument gives a Kochen--Specker
obstruction. The construction is naturally indexed by
$\mathbb Z_p\times\mathbb Z_q$, reflecting the two rulings of the
grids.

From the viewpoint of contextuality, this gives analytic families of
measurement scenarios whose incompatibility structure and parity
obstruction are controlled directly by projective data. In
particular, the distinguished contexts are not obtained by an
exhaustive search through orthogonality relations: they are selected
geometrically by polarity. The construction therefore retains an
explicit description of both the rays and the contextuality proof
throughout the family.

When $p$ and $q$ are relatively prime, the Chinese remainder theorem
identifies $\mathbb Z_p\times\mathbb Z_q$ with $\mathbb Z_{pq}$.
Under this identification, we recover, up to orthogonal equivalence
and reindexing, the four-dimensional Kochen--Specker construction of
Elford and Lison\v{e}k \cite{ElfordLisonek2024}. Thus their two cyclic
orbits acquire a projective interpretation as grids on dual quadrics,
and their distinguished cyclic steps become the two diagonal
steps of the grid. The dual-grid construction is therefore not
merely another parametrization of the cyclic family: it identifies a
projective mechanism underlying it, and continues to make sense when
the single cyclic parametrization is unavailable.

The smallest case of the two cyclic $(n,n)$-grids recovers the Cabello configuration. We
prove that
\[
Z_{3,3}\simeq_{\mathrm{orth}} \mathrm{Cabello}_{18\text{--}9}
\]
where $\simeq_{\mathrm{orth}}$ denotes orthogonal equivalence. Under this identification, the eighteen
Cabello rays split into two nine-point $(3,3)$-grids on
Euclidean-dual quadrics, and the nine contexts of the parity proof are
precisely those selected by the polar-secant construction. Thus the
orthogonality and parity structure of the Cabello configuration admit
a direct interpretation in terms of rulings, secants, and polarity.

The diagonal case $p=q=n$ carries an additional projective structure.
The Euclidean geometry of the supporting quadrics determines two
mutually polar spectral lines, and the incidence geometry of each
cyclic $(n,n)$-grid canonically selects $n$ points on the corresponding
line. Adjoining these points gives a collinear completion
\[
Z_{n,n}\subset\widehat Z_n,
\qquad
|\widehat Z_n|=2n(n+1),
\]
defined for every odd $n\geq3$. The completion is projective in
origin: the new points are forced by collinearity, not by
orthogonality. Nevertheless, after the completion has been performed,
the Euclidean polarity produces a second canonical family of
orthogonal contexts involving the added points. The original
Kochen--Specker obstruction is already present in $Z_{n,n}$; the role
of the completion is instead to reveal additional projective
incidence and, through polarity, additional orthogonality.

The first two diagonal completions recover exceptional root
configurations. For $n=3$, the three points added to each grid form a
trisecant completing it to a projective $D_4$ configuration, and the
two $D_4$ components together give the projectivized $F_4$ root
configuration, equivalently the Peres $24$-ray system. For a
distinguished admissible realization with $n=5$, the same completion
procedure gives the projectivized $H_4$ root configuration. Thus
\[
Z_{3,3}\simeq_{\mathrm{orth}} \mathrm{Cabello}_{18\text{--}9},
\qquad
\widehat Z_3\simeq_{\mathrm{orth}} F_4^{\mathrm{proj}}
                 \simeq_{\mathrm{orth}}\mathrm{Peres}_{24},
\qquad
\widehat Z_5\simeq_{\mathrm{orth}} H_4^{\mathrm{proj}}.
\]
The root systems are therefore not inputs to the construction: they
emerge as exceptional members of an infinite geometric family.

The paper is organized as follows.
Section~\ref{sec:polarity} recalls the projective geometry of Euclidean
polarity, dual quadrics, and orthogonal contexts.
Section~\ref{sec:grids} constructs the dual-grid configurations
$Z_{p,q}$ and proves the parity theorem.
Section~\ref{sec:EL} compares the construction with that of Elford and
Lison\v{e}k.
Section~\ref{sec:cabello} identifies $Z_{3,3}$ with the Cabello
$18$--$9$ configuration.
In Section~\ref{sec:completion} we develop the canonical collinear
completion of the diagonal family and the second family of orthogonal
contexts.
Section~\ref{sec:exceptional} treats the exceptional $F_4$ and $H_4$
realizations.
We conclude with further questions concerning the projective and
orthogonality structures of the completed configurations.
The supplementary material provides explicit coordinates and supporting computational details, with separate SageMath files supplied for reproducibility; the main mathematical arguments are contained in the text.

\section{Euclidean polarity and orthogonality}
\label{sec:polarity}

Let $V=\RR^4$ be endowed with a positive definite inner product $h$,
and let $\PP(V)\simeq\PP^3(\RR)$ be its projectivization. We regard
points of $\PP(V)$ as rays in $V$. Two rays $x=[v]$ and $y=[w]$ are
orthogonal, written $x\perp_h y$, if $h(v,w)=0$.

The inner product identifies $V$ with its dual and induces a polarity
on $\PP(V)$. The polar plane of a point $x=[v]$ is
\[
x^{\perp_h}=\PP(v^{\perp_h}),
\]
while the polar of a projective line $\ell=\PP(U)$ is the projective
line
\[
\ell^{\perp_h}=\PP(U^{\perp_h}).
\]
Here $\dim U=\dim U^{\perp_h}=2$.

\subsection{Dual quadrics}

Let $Q\subset\PP(V)$ be a smooth quadric. After choosing an
$h$-orthonormal basis, write
\[
Q=V(x^tMx),
\]
where $M$ is a real symmetric invertible $4\times4$ matrix, defined
up to a nonzero scalar. The projective dual
$Q^\vee\subset\PP(V^\vee)$ parametrizes the tangent planes to $Q$.
Using the identification $V\simeq V^\vee$ induced by $h$, we regard
$Q^\vee$ as a second quadric in $\PP(V)$ and denote it by $\dual Q$.
In these coordinates,
\[
\dual Q=V(y^tM^{-1}y).
\]

Indeed, the tangent plane to $Q$ at $x$ has equation $x^tMy=0$ and
corresponds under the Euclidean polarity to the ray $[Mx]$. Since
\[
(Mx)^tM^{-1}(Mx)=x^tMx,
\]
this ray belongs to $\dual Q$. Equivalently, a point
$b\in\dual Q$ represents through its polar plane $b^{\perp_h}$ a
tangent plane to $Q$. This is the link between projective duality and
the orthogonality relation used below.

When $Q$ is split, it carries two rulings by projective lines. These
rulings provide the incidence structure of the finite grids introduced
in Section~\ref{sec:grids}, while the Euclidean polarity between $Q$
and $\dual Q$ provides their orthogonality structure.

\subsection{Orthogonal contexts and polar secants}

An \emph{orthogonal context} is a set of four pairwise orthogonal rays
in $\PP(V)$, equivalently the projectivization of an orthogonal basis
of $V$. In the Kochen--Specker setting, these are the maximal
compatible sets of rank-one projectors
\cite{BudroniEtAl2022,KochenSpecker1967}.

The following elementary observation is the local geometric mechanism
behind the dual-grid construction.

\begin{proposition}[Polar-secant principle]
\label{prop:polar-secant}
Let $\ell_A=\langle a_1,a_2\rangle$ and
$\ell_B=\langle b_1,b_2\rangle$ be projective lines such that
\[
\ell_B=\ell_A^{\perp_h}.
\]
If $a_1\perp_h a_2$ and $b_1\perp_h b_2$, then
$\{a_1,a_2,b_1,b_2\}$ is an orthogonal context.
\end{proposition}

\begin{proof}
Every ray on $\ell_B$ is orthogonal to every ray on $\ell_A$, so
$a_i\perp_h b_j$ for all $i,j$. Together with the orthogonality of
the two pairs of the defining points, this gives all six pairwise orthogonality
relations.
\end{proof}

Thus projective polarity supplies four of the six orthogonality
relations defining a context; only the orthogonality within the two
pairs of endpoints has to be imposed separately.

There is also a useful tangent-plane interpretation. If
$a_1,a_2\in Q$ and $b_1,b_2\in\dual Q$ satisfy
\[
\langle b_1,b_2\rangle
=
\langle a_1,a_2\rangle^{\perp_h},
\]
then $b_1^{\perp_h}$ and $b_2^{\perp_h}$ are tangent planes to $Q$,
and their intersection is the secant $\langle a_1,a_2\rangle$.
Thus a secant of $Q$ and its polar secant on $\dual Q$ are two
projective descriptions of the same orthogonality data.

\subsection{Parity configurations}

The polar-secant principle is local: a pair of polar secants produces
a single orthogonal context. Kochen--Specker contextuality is instead
a global property of a family of overlapping contexts. We use the
following standard parity criterion; see, e.g., \cite{WaegellAravind2011Peres}.

\begin{lemma}[Parity criterion]
\label{lem:parity}
Let $Z$ be a finite set of rays and let $\mathcal C$ be a finite
family of orthogonal contexts supported on $Z$. If $|\mathcal C|$ is
odd and every ray of $Z$ occurs in an even number of contexts, then
$Z$ is a Kochen--Specker configuration.
\end{lemma}

\begin{proof}
Suppose that a Kochen--Specker valuation $v\colon Z\to\{0,1\}$
existed. Counting the rays assigned the value $1$ context by context
gives $|\mathcal C|$, which is odd. Counting the same incidences ray
by ray gives
\[
\sum_{x\in Z}m(x)v(x),
\]
where $m(x)$ is the number of contexts containing $x$. This sum is
even because every $m(x)$ is even, a contradiction.
\end{proof}

Thus the construction has two distinct levels: Euclidean polarity
provides the local orthogonality relations needed to form contexts,
while the global incidence pattern of those contexts produces the
Kochen--Specker obstruction through parity. Section~\ref{sec:grids}
realizes these two mechanisms simultaneously.

\section{Dual-grid Kochen--Specker configurations}
\label{sec:grids}

We now construct the first family announced in the Introduction. For
every pair of odd integers $p,q\geq3$, we shall obtain two finite grids
on Euclidean-dual quadrics whose union consists of $2pq$ rays
supporting $pq$ distinguished orthogonal contexts, with every ray
occurring in exactly two contexts. The parity criterion of
Lemma~\ref{lem:parity} will then give a Kochen--Specker configuration.

The geometric mechanism is the one isolated in
Section~\ref{sec:polarity}. The relative position of the two grids is
chosen so that secants in two distinguished diagonal steps occur
in polar pairs. Proposition~\ref{prop:polar-secant} turns these pairs
into orthogonal contexts, while their global incidence pattern
produces the parity obstruction.

\subsection{Finite grids and a dual normal form}
\label{subsec:finite-grids}

Let $U$ and $W$ be two-dimensional real vector spaces. The
projectivized decomposable tensors form the Segre quadric
$\Sigma\subset\PP(U\otimes W)$, naturally identified with
$\PP(U)\times\PP(W)$. After choosing bases and identifying
$U\otimes W\simeq\operatorname{Mat}_2(\RR)$, a tensor $X$ is
decomposable if and only if the corresponding matrix has rank one.
Thus the quadric $\Sigma$ is given by $\det X=0$.

The two rulings have a simple tensor interpretation. Fixing
$[u]\in\PP(U)$ and varying $[w]\in\PP(W)$ gives a line of one ruling,
while fixing $[w]$ and varying $[u]$ gives a line of the other.
Accordingly, finite subsets
$S=\{[u_0],\ldots,[u_{p-1}]\}\subset\PP(U)$ and
$T=\{[w_0],\ldots,[w_{q-1}]\}\subset\PP(W)$ determine a
$(p,q)$-\emph{grid}
\[
\Gamma(S,T)=
\{[u_i\otimes w_j]:(i,j)\in\ZZ_p\times\ZZ_q\}\subset\Sigma.
\]
Its $pq$ points are the intersections of $p$ lines of one ruling and
$q$ lines of the other.

We use cyclic grids, obtained by taking a finite orbit of a projective
rotation in each factor and forming their product. Such a grid is
naturally indexed by $\ZZ_p\times\ZZ_q$, with the two factors
corresponding to the two rulings. No coprimality assumption is involved in this product
description. When $p$ and $q$ are coprime, the Chinese remainder
theorem replaces it by a single cyclic parameter; this will give the
connection with the construction of Elford and Lison\v{e}k in
Section~\ref{sec:EL}.

For $\theta\in\RR$, let
\[
R_\theta=
\begin{pmatrix}
	\cos\theta&-\sin\theta\\
	\sin\theta&\cos\theta
\end{pmatrix},
\qquad
u(\theta)=(\cos\theta,\sin\theta)^t.
\]
In matrix coordinates, rotations on the two factors act independently
on the Segre quadric by $X\mapsto R_\alpha X R_\beta^t$. We use an
orthogonal change of coordinates which separates this tensor-product
action into two planar rotations and makes its interaction with the
Euclidean metric explicit. The tensor-product description is used
only to realize the split quadric as a Segre surface and to describe
its two rulings; no bipartite interpretation is assumed.

Let $p,q\geq3$ be odd, choose
units $a\in\ZZ_p^\times$ and $b\in\ZZ_q^\times$, and set
\[
\alpha=\frac{2\pi a}{p},\qquad
\beta=\frac{2\pi b}{q},\qquad
\theta_{ij}^{\pm}=i\alpha\pm j\beta.
\]
Under the identification
$U\otimes W\simeq\operatorname{Mat}_2(\RR)$, followed by the
orthogonal change of coordinates
\[
X\longmapsto
(x_{11}+x_{22},\,x_{21}-x_{12},\,
 x_{11}-x_{22},\,x_{21}+x_{12}),
\]
up to a common factor, the decomposable tensor
$u(i\alpha)\otimes u(j\beta)$ is sent to
$(u(\theta_{ij}^{-}),u(\theta_{ij}^{+}))$. Thus the tensor-product
action becomes the direct sum of rotations through the angles
$\alpha-\beta$ and $\alpha+\beta$.

Fix $c\in\RR^\times$ and use the resulting orthogonal decomposition
$V\simeq\RR^2\oplus\RR^2$.  Define
\[
a_{ij}=
\bigl(u(\theta_{ij}^{-}),-c\,u(\theta_{ij}^{+})\bigr),
\qquad
b_{ij}=
\bigl(c\,u(\theta_{ij}^{-}),u(\theta_{ij}^{+})\bigr),
\]
and let $A_{ij}=[a_{ij}]$, $B_{ij}=[b_{ij}]$,
$A=\{A_{ij}\}$, and $B=\{B_{ij}\}$.

Write $(u_1,u_2)$ and $(v_1,v_2)$ for the coordinates on the two
summands of $V\simeq\RR^2\oplus\RR^2$.
\begin{proposition}
\label{prop:dual-grid-quadrics}
The sets $A$ and $B$ are cyclic $(p,q)$-grids. Moreover, $A$ lies
on the smooth quadric
\[
Q_c\colon
c^2(u_1^2+u_2^2)-(v_1^2+v_2^2)=0,
\]
while $B$ lies on its Euclidean dual $\dual{Q_c}$, represented by
$(u_1^2+u_2^2)-c^2(v_1^2+v_2^2)=0$.
\end{proposition}

\begin{proof}
After the orthogonal change of coordinates above, the Segre quadric is
$\Sigma_0\colon\|u\|^2-\|v\|^2=0$. The sets $A$ and $B$ are its
images under the projective transformations
$T_A(u,v)=(u,-cv)$ and $T_B(u,v)=(cu,v)$, respectively. Hence both
are cyclic $(p,q)$-grids.

Moreover, $T_A(\Sigma_0)=Q_c$, while $T_B(\Sigma_0)$ is represented
by $(u_1^2+u_2^2)-c^2(v_1^2+v_2^2)=0$. The matrix of $Q_c$ is
$M_c=\operatorname{diag}(c^2,c^2,-1,-1)$, and
$c^2M_c^{-1}=\operatorname{diag}(1,1,-c^2,-c^2)$. By
Section~\ref{sec:polarity}, the second quadric is therefore
$\dual{Q_c}$.
\end{proof}

The parameter $c$ leaves the projective incidence structure of the
two grids unchanged but controls their relative Euclidean geometry.
We choose it so that the two diagonal steps of the product
indexing give orthogonal pairs. The required condition is
\begin{equation}
	\label{eq:c-condition}
	c^2=-\frac{\cos(\alpha-\beta)}{\cos(\alpha+\beta)},
\end{equation}
whenever the right-hand side is positive and finite. Thus, once the
two cyclic steps are fixed, $c^2$ is determined by orthogonality.

\subsection{Polar secants and orthogonal contexts}
\label{subsec:dual-grid-contexts}

The two diagonal steps $(1,1)$ and $(1,-1)$ in
$\ZZ_p\times\ZZ_q$ produce the polar secants underlying the
construction.

\begin{proposition}
\label{prop:dual-grid-polar-secants}
Assume that \eqref{eq:c-condition} holds. For every
$(i,j)\in\ZZ_p\times\ZZ_q$, the secants
$\ell_{ij}=\langle A_{ij},A_{i-1,j-1}\rangle$ and
$m_{ij}=\langle B_{i,j-1},B_{i-1,j}\rangle$ are polar to one another,
and the two endpoints on each secant are orthogonal.
\end{proposition}

\begin{proof}
The scalar products within the two secants coincide:
\[
\langle a_{ij},a_{i-1,j-1}\rangle
=
\langle b_{i,j-1},b_{i-1,j}\rangle
=
\cos(\alpha-\beta)+c^2\cos(\alpha+\beta),
\]
and hence vanish by \eqref{eq:c-condition}. The four cross scalar
products vanish independently of this condition. For example,
$\langle a_{ij},b_{i,j-1}\rangle
=c\cos\beta-c\cos\beta=0$, and the other three follow by the
corresponding shifts. Hence every ray on $m_{ij}$ is orthogonal to
every ray on $\ell_{ij}$, so $m_{ij}=\ell_{ij}^{\perp_h}$.
\end{proof}

\begin{corollary}
\label{cor:dual-grid-contexts}
For every $(i,j)\in\ZZ_p\times\ZZ_q$, the four rays
\[
C_{ij}=
\{A_{ij},A_{i-1,j-1},B_{i,j-1},B_{i-1,j}\}
\]
form an orthogonal context.
\end{corollary}

\begin{proof}
This follows immediately from
Proposition~\ref{prop:dual-grid-polar-secants} and the polar-secant
principle, Proposition~\ref{prop:polar-secant}.
\end{proof}

We next record the elementary facts needed for the parity count.

\begin{lemma}
\label{lem:dual-grid-distinct}
If $p$ and $q$ are odd, the rays within each grid are pairwise
distinct. If moreover $c^2\neq1$, then $A\cap B=\varnothing$ and
hence $|A\sqcup B|=2pq$.
\end{lemma}

\begin{proof}
Suppose that $A_{ij}=A_{rs}$. Since all representatives $a_{ij}$ have
the same norm, $a_{ij}=\pm a_{rs}$. Comparing the two planar
components shows that
$(i-r)\alpha-(j-s)\beta$ and $(i-r)\alpha+(j-s)\beta$ are both
multiples of $\pi$. Hence $(i-r)\alpha$ and $(j-s)\beta$ are
multiples of $\pi$. Since $p$ and $q$ are odd and $a,b$ are units
modulo $p,q$, respectively, this forces $i=r$ in $\ZZ_p$ and $j=s$
in $\ZZ_q$. The argument for $B$ is identical.

If a real point belonged to both $Q_c$ and $\dual{Q_c}$, their
equations would give
$(c^2-1)(u_1^2+u_2^2)=(c^2-1)(v_1^2+v_2^2)=0$. For $c^2\neq1$ this
forces both planar components to vanish. Thus
$Q_c(\RR)\cap\dual{Q_c}(\RR)=\varnothing$, and in particular
$A\cap B=\varnothing$.
\end{proof}

\subsection{The dual-grid family}
\label{subsec:dual-grid-family}

It remains to show that \eqref{eq:c-condition} can be satisfied for
every pair of odd integers $p,q\geq3$. For an odd integer $x\geq3$,
set
\[
k_x=
\begin{cases}
(x-1)/4,&x\equiv1\pmod4,\\
(x+1)/4,&x\equiv3\pmod4.
\end{cases}
\]
We call a pair of cyclic steps $(a,b)$ \emph{admissible} if $a$ and
$b$ are units modulo $p$ and $q$, respectively, the right-hand side
of \eqref{eq:c-condition} is positive and finite, and the resulting
value of $c$ satisfies $c^2\neq1$.

For an admissible pair $(a,b)$, let $c>0$ be determined by
\eqref{eq:c-condition}, and denote the resulting dual-grid
configuration by
\[
Z_{p,q}(a,b)=A\sqcup B.
\]
When the choice of cyclic steps is fixed or irrelevant, we abbreviate
this notation to $Z_{p,q}$.
\begin{lemma}
\label{lem:admissible-parameters}
Let $p,q\geq3$ be odd. The choice $a=k_p$ and $b=k_q$ is admissible.
\end{lemma}

\begin{proof}
Write $x=4m+1$ or $x=4m+3$. In the first case $k_x=m$ and
$\gcd(m,4m+1)=1$, while in the second $k_x=m+1$ and
$\gcd(m+1,4m+3)=1$. Thus $k_x$ is a unit modulo $x$.

Moreover,
$2\pi k_x/x=\pi/2\mp\pi/(2x)$, with the minus sign when
$x\equiv1\pmod4$ and the plus sign when $x\equiv3\pmod4$. Hence
$\alpha$ and $\beta$ both lie within $\pi/6$ of $\pi/2$, so
$\cos(\alpha-\beta)>0$ and $\cos(\alpha+\beta)<0$. Thus
\eqref{eq:c-condition} determines a real nonzero value of $c$.

Finally, if $c^2=1$, then
$\cos(\alpha-\beta)+\cos(\alpha+\beta)=0$, hence
$2\cos\alpha\cos\beta=0$, which is impossible for these angles.
\end{proof}

The remaining point is global: the contexts
$C_{ij}$ must have the incidence pattern required by the parity
criterion.

\begin{lemma}
\label{lem:dual-grid-incidence}
The contexts $C_{ij}$ are pairwise distinct, and every ray of
$A\sqcup B$ belongs to exactly two of them. More precisely,
$A_{ij}\in C_{ij}\cap C_{i+1,j+1}$ and
$B_{ij}\in C_{i,j+1}\cap C_{i+1,j}$, and these are the only contexts
containing the respective rays.
\end{lemma}

\begin{proof}
The incidence formulas follow directly from the definition of
$C_{ij}$, and the two contexts containing a given ray are distinct
because $p,q\geq3$.

Suppose that $C_{ij}=C_{rs}$. Comparing their intersections with $A$
gives
$\{A_{ij},A_{i-1,j-1}\}
=\{A_{rs},A_{r-1,s-1}\}$.
By Lemma~\ref{lem:dual-grid-distinct}, either $(i,j)=(r,s)$ or the
two entries are exchanged. The latter would imply $2=0$ in both
$\ZZ_p$ and $\ZZ_q$, impossible since $p$ and $q$ are odd. Hence
$(i,j)=(r,s)$.
\end{proof}

We can now state the first main result.

\begin{theorem}[Dual-grid Kochen--Specker configurations]
	\label{thm:dual-grid-KS}
	For every pair of odd integers $p,q\geq3$, there exists an admissible
	pair $(a,b)$ such that the dual-grid configuration $Z_{p,q}(a,b)$
	consists of $2pq$ rays and supports $pq$ orthogonal contexts. Every
	ray occurs in exactly two of these contexts. In particular,
	$Z_{p,q}(a,b)$ is a Kochen--Specker configuration admitting a parity
	proof.
\end{theorem}

\begin{proof}
	Choose $a=k_p$ and $b=k_q$ as in
	Lemma~\ref{lem:admissible-parameters}, and let $c$ be determined by
	\eqref{eq:c-condition}. Proposition~\ref{prop:dual-grid-quadrics}
	gives the two grids on Euclidean-dual quadrics, and
	Corollary~\ref{cor:dual-grid-contexts} gives the $pq$ orthogonal
	contexts $C_{ij}$. By Lemma~\ref{lem:dual-grid-distinct},
	$Z_{p,q}(a,b)$ contains $2pq$ distinct rays, while
	Lemma~\ref{lem:dual-grid-incidence} shows that every ray occurs in
	exactly two contexts. Since $pq$ is odd, Lemma~\ref{lem:parity}
	gives the required Kochen--Specker obstruction.
\end{proof}

\begin{remark}
\label{rem:no-coprimality}
No coprimality assumption on $p$ and $q$ is required. The natural
indexing set is $\ZZ_p\times\ZZ_q$, reflecting the two rulings of the
grids. When $\gcd(p,q)=1$, the Chinese remainder theorem identifies
this product with $\ZZ_{pq}$. Section~\ref{sec:EL} shows that in this
case the construction recovers that of Elford and Lison\v{e}k.
\end{remark}

\begin{remark}
\label{rem:base-quartic}
For $c^2\neq1$, the real quadrics $Q_c$ and $\dual{Q_c}$ are
disjoint, whereas their complexifications meet in a reducible
quartic. Indeed, their equations imply
$u_1^2+u_2^2=v_1^2+v_2^2=0$, so over $\CC$ their intersection is the
union of the four lines obtained by independently choosing the signs
in $u_1\pm\sqrt{-1}\,u_2=0$ and
$v_1\pm\sqrt{-1}\,v_2=0$.
\end{remark}

\begin{remark}[Dependence on the cyclic steps]
	\label{rem:metric-parameters}
	For fixed $p$, $q$, and $c$, changing the units $a$ and $b$ only
	reindexes the points of the two grids. In the orthogonal construction,
	however, $c$ is itself determined by $(a,b)$ through
	\eqref{eq:c-condition}. Thus different admissible pairs may determine
	different Euclidean realizations $Z_{p,q}(a,b)$ of the same abstract
	dual-grid incidence structure. These realizations need not be
	orthogonally equivalent, and their full orthogonality hypergraphs may
	differ.
\end{remark}

The diagonal case $p=q$ carries additional geometry which is not
needed for the parity construction of
Theorem~\ref{thm:dual-grid-KS}. We return to it in
Section~\ref{sec:completion}, after relating the dual-grid family to
known four-dimensional Kochen--Specker configurations.

\section{The Elford--Lison\v{e}k construction}
\label{sec:EL}

We now compare the dual-grid family with the four-dimensional
Kochen--Specker construction of Elford and Lison\v{e}k
\cite{ElfordLisonek2024}. Their construction is described by two
cyclic orbits of length $pq$, under the assumption that $p$ and $q$
are coprime, whereas the dual-grid construction is naturally indexed
by $\ZZ_p\times\ZZ_q$, reflecting the two rulings of the supporting
quadrics. We show that in the coprime case the two realizations are
orthogonally equivalent. Under this equivalence, the two distinguished
cyclic steps of Elford and Lison\v{e}k become the two diagonal
steps of the grids.

Assume throughout this section that $p,q\geq3$ are odd and coprime,
and put $d=pq$. The Chinese remainder theorem gives an isomorphism
\[
\ZZ_d\simeq\ZZ_p\times\ZZ_q,
\qquad
t\longmapsto(t\bmod p,t\bmod q).
\]
Thus the single cyclic parameter may equivalently be regarded as the pair $(i,j)$ indexing the two rulings of the grid. 

The two diagonal steps of the grid become precisely the two distinguished cyclic steps. In particular, the apparently
one-dimensional cyclic structure retains the product geometry of the two rulings: the residue $r$ simply encodes the second diagonal
step in a single cyclic coordinate.

We now compare the Euclidean realizations. Choose units
$a\in\ZZ_p^\times$ and $b\in\ZZ_q^\times$, and put
$\alpha=2\pi a/p$ and $\beta=2\pi b/q$. In the notation of
Elford and Lison\v{e}k, the cyclic action is generated by
$M=R_\alpha\otimes R_\beta$, and the two generating vectors may be
taken as
\[
a_{\mathrm{EL}}=
\begin{pmatrix}
(1-c)\cos\beta\\
(1-c)\sin\beta\\
-(1+c)\sin\beta\\
(1+c)\cos\beta
\end{pmatrix},
\qquad
b_{\mathrm{EL}}=
\begin{pmatrix}
c+1\\
0\\
0\\
c-1
\end{pmatrix}.
\]
Their parameter $c$ satisfies the same relation as in
\eqref{eq:c-condition}, 
$c^2=-\frac{\cos(\alpha-\beta)}{\cos(\alpha+\beta)}.$

Let
\[
\Psi(x_{11},x_{12},x_{21},x_{22})
=
(x_{11}+x_{22},\,x_{21}-x_{12},\,
x_{11}-x_{22},\,x_{21}+x_{12})
\]
denote the change of coordinates used in
Section~\ref{sec:grids}. Since $\Psi/\sqrt2$ is orthogonal and
\[
\Psi M\Psi^{-1}
=
R_{\alpha-\beta}\oplus R_{\alpha+\beta},
\]
it intertwines the tensor-product action defining the
Elford--Lison\v{e}k orbits with the product rotation defining our
grids. A direct computation gives
\[
\begin{aligned}
	\Psi(a_{\mathrm{EL}})
	&=
	2(\cos\beta,-\sin\beta,-c\cos\beta,-c\sin\beta)
	=2a_{0,1},\\
	\Psi(b_{\mathrm{EL}})
	&=
	(2c,0,2,0)
	=2b_{0,0}.
\end{aligned}
\] Hence, for every $t\in\ZZ_d$,
\[
\Psi(M^t a_{\mathrm{EL}})=2a_{t,t+1},
\qquad
\Psi(M^t b_{\mathrm{EL}})=2b_{t,t},
\]
where the first and second indices are read modulo $p$ and $q$,
respectively. Since $p$ and $q$ are coprime, the Chinese remainder
theorem shows that the two cyclic orbits are exactly the two grids
$A$ and $B$, up to the indicated reindexing.

We can now identify the distinguished contexts as well.

\begin{proposition}
\label{prop:EL-equivalence}
Let $p,q\geq3$ be odd and relatively prime. Then the
Elford--Lison\v{e}k configuration associated with $p$ and $q$
\cite{ElfordLisonek2024} is orthogonally equivalent, up to
reindexing, to the dual-grid configuration $Z_{p,q}$ of
Theorem~\ref{thm:dual-grid-KS}. Under this equivalence, the two cyclic
orbits become the two $(p,q)$-grids
$A\subset Q_c$ and $B\subset\dual{Q_c}$, and the distinguished
orthogonal bases become the contexts $C_{ij}$ of
Corollary~\ref{cor:dual-grid-contexts}.
\end{proposition}

\begin{proof}
The preceding computation shows that the orthogonal transformation
$\Psi/\sqrt2$ maps the two Elford--Lison\v{e}k orbits onto $A$ and
$B$, up to reindexing.

It remains to compare the contexts. Let $r\in\ZZ_d$ be determined by
$r\equiv1\pmod p$ and $r\equiv-1\pmod q$. The distinguished
Elford--Lison\v{e}k basis indexed by $t$ is 
$$\{M^t a_{\mathrm{EL}},M^{t-1}a_{\mathrm{EL}},
  M^t b_{\mathrm{EL}},M^{t-r}b_{\mathrm{EL}}\}.$$ Writing $i=t\bmod p$ and $j=t\bmod q$, its image under $\Psi$ is 
$\{A_{i,j+1},A_{i-1,j},
  B_{i,j},B_{i-1,j+1}\}
=
C_{i,j+1}.$ Thus the same orthogonal transformation identifies both the rays and
the distinguished orthogonal contexts.
\end{proof}

\begin{remark}
\label{rem:EL-coprimality}
The coprimality assumption is needed only to identify
$\ZZ_p\times\ZZ_q$ with the single cyclic parameter space
$\ZZ_{pq}$. It is not intrinsic to the grid construction, whose two
factors reflect the two rulings of the supporting quadrics. Thus
Theorem~\ref{thm:dual-grid-KS} extends the Elford--Lison\v{e}k family
from relatively prime odd $p,q$ to arbitrary odd $p,q\geq3$.
\end{remark}

\section{The Cabello configuration}
\label{sec:cabello}

We now consider the smallest member of the dual-grid family.
For $p=q=3$ there is only one cyclic step up to sign, so we take
$a=b=1$. Thus $\alpha=\beta=2\pi/3$, and
\eqref{eq:c-condition} gives $c^2=2$. We choose $c=\sqrt2$.
Hence $Z_{3,3}=A\sqcup B$ consists of two $(3,3)$-grids on the
Euclidean-dual quadrics $Q_{\sqrt2}$ and $\dual{Q_{\sqrt2}}$.
By Theorem~\ref{thm:dual-grid-KS}, it consists of $18$ rays and
supports nine distinguished contexts, with every ray occurring
exactly twice.

We show that $Z_{3,3}$ is the classical $18$-ray configuration of
Cabello, Estebaranz and Garc\'ia-Alcaine
\cite{CabelloEstebaranzGarciaAlcaine1996}. For the identification it
is convenient to use its realization inside the Peres $24$-ray
configuration \cite{Peres1991,WaegellAravind2011Peres}.

\subsection{Cabello inside the Peres configuration}
\label{subsec:cabello-peres}

The projective Peres configuration decomposes as $P=D\sqcup D'$,
where $D$ and $D'$ are two projectivized root configurations of type
$D_4$. Each component contains trisecants, and suitable pairs
$L\subset D$ and $L'\subset D'$ are mutually orthogonal, in the
sense that every ray of $L$ is orthogonal to every ray of $L'$.
Deleting the six rays of such a pair gives an $18$-ray, $9$-context
parity proof contained in the Peres configuration~\cite{WaegellAravind2011Peres}.

The projective feature relevant here is that the complement of a
trisecant in either $D_4$ component is a $(3,3)$-grid. Thus the
decomposition of the resulting $18$-ray configuration into
$D\setminus L$ and $D'\setminus L'$ has precisely the same incidence
structure as the decomposition $Z_{3,3}=A\sqcup B$.

\begin{lemma}
\label{lem:D4-minus-trisecant}
Let $D$ be the projectivized $D_4$ root configuration and let
$L\subset D$ be a trisecant. Then the nine points of $D\setminus L$
are the intersection points of three lines of one ruling and three
lines of the other ruling of a smooth quadric. In particular,
$D\setminus L$ is a $(3,3)$-grid.
\end{lemma}

\begin{proof}
Using the explicit Peres coordinates recorded in the supplementary
material (Section \ref{app:cabello}), one checks directly that the nine points of $D\setminus L$
lie on six trisecants, three in each ruling, and hence form a
$(3,3)$-grid.

There is also a conceptual proof using geproci configurations. Recall
that an $(a,b)$-geproci set in $\PP^3$ is a finite set of $ab$ points
whose general projection to $\PP^2$ is a complete intersection of
plane curves of degrees $a$ and $b$. The projectivized $D_4$ root
configuration is a $(3,4)$-geproci set
\cite[Chapter 3]{ChiantiniEtAl2022}. Removing the three points on a
trisecant leaves a $(3,3)$-geproci configuration
\cite[Lemma 4.5]{ChiantiniEtAl2022}. In characteristic zero, every
$(3,3)$-geproci configuration is a grid
\cite[Theorem 5.12]{chiantini2021sets}.
\end{proof}

We can now identify the two Euclidean configurations.

\begin{theorem}
\label{thm:cabello-identification}
The dual-grid configuration $Z_{3,3}$ is orthogonally equivalent to
the Cabello $18$--$9$ configuration. More precisely, there are
mutually polar trisecants $L\subset D$ and $L'\subset D'$ in the two
$D_4$ components of the Peres configuration such that
\[
Z_{3,3}\simeq_{\mathrm{orth}}
(D\setminus L)\sqcup(D'\setminus L').
\]
Under this equivalence, the two grids are supported on
Euclidean-dual quadrics, and the contexts $C_{ij}$ correspond to the
nine contexts of the Cabello parity proof.
\end{theorem}

\begin{proof}
We use the standard Peres realization recorded in
Section~\ref{app:cabello} of the Supplementary Material and choose 
$L=\{P_1,P_5,P_{12}\}$ and
$L'=\{P_{16},P_{18},P_{24}\}$.
These trisecants are mutually polar. By
Lemma~\ref{lem:D4-minus-trisecant}, their complements are
$(3,3)$-grids.

For $Z_{3,3}$, with $\alpha=\beta=2\pi/3$ and $c=\sqrt2$, consider
the matrix
\[
R=
\begin{pmatrix}
0&1&0&0\\
\frac1{\sqrt3}&0&-\sqrt{\frac23}&0\\
\frac1{\sqrt3}&0&\frac1{\sqrt6}&-\frac1{\sqrt2}\\
\frac1{\sqrt3}&0&\frac1{\sqrt6}&\frac1{\sqrt2}
\end{pmatrix}.
\]
Its columns form an orthonormal basis, so $R\in O(4)$. Direct
substitution gives
$R(A)=D\setminus L$ and $R(B)=D'\setminus L'$.
Under the same transformation, the nine contexts $C_{ij}$ are mapped
onto the nine orthogonal bases of the corresponding Cabello
$18$--$9$ parity proof. Hence $R$ gives the claimed orthogonal
equivalence.

The complete ray correspondence and an exact verification are
provided in the supplementary material.
\end{proof}

Theorem~\ref{thm:cabello-identification} therefore identifies the
Cabello configuration not merely as an isolated $18$-ray example,
but as the first member of the dual-grid family. Its
eighteen rays split into two $(3,3)$-grids on Euclidean-dual
quadrics, and its nine contexts are precisely the polar-secant
contexts of the dual-grid construction.

The trisecants removed in this description point to additional
geometry. Each $(3,3)$-grid is completed by the three points of its
trisecant to a projectivized $D_4$ configuration. Rather than treating
this as a special feature of the Cabello configuration, we shall see
in Section~\ref{sec:completion} that the same phenomenon extends
canonically to every  configuration $Z_{n,n}$.

\section{Collinear completions of diagonal dual grids}
\label{sec:completion}

We now specialize to the diagonal family $Z_{n,n}$, with $n\geq3$
odd. In this case the two rulings of each grid determine two
distinguished lines, intrinsic to the Euclidean pair consisting of
the supporting quadric and the ambient polarity. These spectral lines
support distinguished $n$-point orbits which are forced by the
collinear incidence of the grid.

This leads to a canonical completion
\[
\widehat Z_n=Z_{n,n}\sqcup X_n\sqcup X'_n
\]
by $2n$ additional rays. As above, the notation suppresses the
dependence on the admissible cyclic steps $(a,b)$; accordingly,
$\widehat Z_n$ denotes the completion of the chosen realization of
$Z_{n,n}$. Although the completion is defined projectively, through
collinear incidence, the added rays reveal a second canonical family
of orthogonal contexts.

\subsection{Spectral lines}
\label{subsec:spectral-lines}

Let $Q\subset\PP(V)$ be a smooth real quadric, represented by a
nondegenerate symmetric matrix $M$, and let $h$ be the Euclidean
inner product. Since $M$ is self-adjoint with respect to $h$, its
eigenspaces are mutually orthogonal.

For the normal form
\[
Q_c\colon
c^2(u_1^2+u_2^2)-(v_1^2+v_2^2)=0
\]
of Section~\ref{sec:grids}, the corresponding self-adjoint operator
has two two-dimensional eigenspaces,
\[
V_+=\RR^2\oplus0,
\qquad
V_-=0\oplus\RR^2.
\]
Their projectivizations
\[
L_+=\PP(V_+),
\qquad
L_-=\PP(V_-)
\]
will be called the \emph{spectral lines} of $Q_c$. They are mutually
polar with respect to the Euclidean polarity:
$L_-=L_+^{\perp_h}$.

The definition is intrinsic. If a smooth real quadric $Q$ is
represented, up to scale, by a self-adjoint operator with two real
eigenvalues of multiplicity two, then the projectivizations of the
two eigenspaces form an unordered pair of spectral lines. This pair
is preserved by orthogonal changes of coordinates.

For $Z_{n,n}$, these lines interact with the two rulings
in a particularly rigid way. We make this explicit for one ruling;
the other is obtained by interchanging the two orthogonal summands.

Let $G_n=\{A_{ij}\}\subset Q_c$ be a cyclic $(n,n)$-grid, with
\[
\alpha=\frac{2\pi a}{n},
\qquad
\beta=\frac{2\pi b}{n},
\qquad
a,b\in\ZZ_n^\times.
\]
For the ruling obtained by fixing the first index, define
\[
P_r=
\left[
u\left(\frac{\pi}{2}+\frac{2\pi r}{n}\right):0
\right],
\qquad r\in\ZZ_n,
\]
and set $X_n=\{P_r:r\in\ZZ_n\}\subset L_+$.

\begin{lemma}[Transversals through the spectral line]
	\label{lem:spectral-transversals}
	Let $R_0,\ldots,R_{n-1}$ be the lines of the ruling of $G_n$
	obtained by fixing the first index. For $i\neq k$, the line through
	$A_{ij}\in R_i$ and $A_{k\ell}\in R_k$ meets $L_+$ if and only if
	\[
	ai+bj\equiv ak+b\ell\pmod n.
	\]
	In this case its intersection with $L_+$ is $P_{ak-bj}$.
	
	Conversely, for each $P_r\in X_n$ and  $i\neq k$, there is a
	unique transversal through $P_r$ meeting $R_i$ and $R_k$, and its
	intersection points with these two ruling lines belong to the grid.
\end{lemma}

\begin{proof}
	Write
	\[
	A_{ij}
	=
	[u(i\alpha-j\beta):-c\,u(i\alpha+j\beta)].
	\]
	The line $\langle A_{ij},A_{k\ell}\rangle$ meets $L_+$ precisely
	when a nonzero linear combination of the two representatives has
	vanishing $V_-$ component. Equivalently,
	$u(i\alpha+j\beta)$ and $u(k\alpha+\ell\beta)$ are proportional.
	Since $n$ is odd, this is equivalent to
	\[
	ai+bj\equiv ak+b\ell\pmod n.
	\]
	
	Under this condition the intersection with $L_+$ is represented, up
	to a nonzero scalar, by
	$u(i\alpha-j\beta)-u(k\alpha-\ell\beta)$. Using
	\[
	u(x)-u(y)
	=
	2\sin\left(\frac{x-y}{2}\right)
	u\left(\frac{x+y}{2}+\frac{\pi}{2}\right),
	\]
	its direction is
	\[
	\frac{\pi}{2}
	+
	\frac{\alpha(i+k)-\beta(j+\ell)}{2}.
	\]
	The congruence above gives
	$a(i+k)-b(j+\ell)\equiv2(ak-bj)\pmod n$, so this direction is
	projectively equal to
	$\pi/2+2\pi(ak-bj)/n$. Hence the intersection point is
	$P_{ak-bj}$.
	
	Conversely, fix $P_r\in X_n$ and $i\neq k$, and set
	\[
	j=b^{-1}(ak-r),
	\qquad
	\ell=b^{-1}(ai-r)
	\]
	in $\ZZ_n$. Then
	$ai+bj=ak+b\ell$ and $ak-bj=r$. By the first part,
	$P_r,A_{ij},A_{k\ell}$ are collinear. The displayed formulas
	determine $j$ and $\ell$ uniquely, and hence the transversal is
	unique.
\end{proof}

The lemma shows that the points of $X_n$ are not an additional
choice: they are forced by the transversals of the grid.

\subsection{Collinear completion of a diagonal grid}
\label{subsec:collinear-completion}

Following \cite{chiantini2023combinatorics}, we say that a finite point set $Z$ is collinearly complete with
respect to a finite line arrangement $\mathcal L$ if every line
meeting at least three members of $\mathcal L$ and containing a
point of $Z$ contains in $Z$ all its intersection points with
$\mathcal L$.

\begin{theorem}[Collinear completion of $Z_{n,n}$]
	\label{thm:collinear-completion}
	Let $n\geq3$ be odd and let $G_n\subset Q_c$ be a cyclic
	$(n,n)$-grid occurring in $Z_{n,n}$. Choose one of its two rulings,
	with supporting lines
	$\mathcal R=\{R_0,\ldots,R_{n-1}\}$.
	
	For the ruling obtained by fixing the first index, let $L=L_+$ and
	let $X_n\subset L_+$ be the orbit of
	Lemma~\ref{lem:spectral-transversals}. For the opposite ruling, let
	$L=L_-$ and define the analogous $n$-point orbit by interchanging
	the two orthogonal summands.
	
	Then the corresponding $n$-point orbit on $L$ is the unique smallest
	set $X$ such that $G_n\cup X$ is collinearly complete with respect to
	$\mathcal R\cup\{L\}$. In particular, the completion consists of
	exactly $n$ points and is independent of the spectral parameter $c$.
\end{theorem}

\begin{proof}
	We first consider the ruling of
	Lemma~\ref{lem:spectral-transversals}, for which $L=L_+$. By that
	lemma, every transversal through a grid point and meeting two
	distinct lines of $\mathcal R$ meets $L_+$ in a point of $X_n$.
	Conversely, every point of $X_n$ occurs in this way. Since $b$ is a
	unit modulo $n$, for fixed $k$ the residues $ak-bj$, with
	$j\in\ZZ_n$, run through all of $\ZZ_n$. Hence these intersection
	points are precisely the $n$ distinct points of $X_n$.
	
	We next show that $G_n\cup X_n$ is collinearly complete. Let $T$ be
	a line meeting at least three members of
	$\mathcal R\cup\{L_+\}$ and containing a point of $G_n\cup X_n$.
	Since the arrangement consists of the lines of $\mathcal R$ together
	with the single line $L_+$, there are only two cases: either $T$
	meets at least three lines of $\mathcal R$, or it meets $L_+$ and at
	least two lines of $\mathcal R$.
	
	Suppose first that $T$ meets at least three lines of $\mathcal R$.
	Since the lines of a ruling are pairwise disjoint, these intersections
	give at least three distinct points of $Q_c$. A line not contained in
	a smooth quadric meets it in at most two points, so $T\subset Q_c$
	and hence $T$ belongs to the opposite ruling. Since
	$L_+\cap Q_c=\varnothing$, the line $T$ does not meet $L_+$. Moreover,
	if $T$ contains a point of $G_n$, it is the line of the opposite
	ruling through that grid point, hence one of the $n$ selected lines
	of that ruling. All its intersections with the lines of $\mathcal R$
	therefore belong to $G_n$. The alternative that $T$ contains a point
	of $X_n$ cannot occur, since $X_n\subset L_+$ and
	$T\cap L_+=\varnothing$.
	
	Suppose instead that $T$ meets $L_+$ and at least two distinct lines
	$R_i,R_k$. Since $L_+\cap Q_c=\varnothing$, the line $T$ is not
	contained in $Q_c$ and therefore meets $Q_c$ in at most two points.
	Hence $R_i$ and $R_k$ are the only members of $\mathcal R$ met by
	$T$. If $T$ contains a point of $G_n$, its intersection with $L_+$
	belongs to $X_n$ by Lemma~\ref{lem:spectral-transversals}. If it
	contains a point of $X_n$, the converse part of the same lemma shows
	that its intersections with $R_i$ and $R_k$ are grid points. Thus in
	either case all intersections of $T$ with
	$\mathcal R\cup\{L_+\}$ belong to $G_n\cup X_n$.
	
	It remains to prove that the completion is forced. Every
	$X_r\in X_n$ lies on a transversal meeting $L_+$ and two distinct
	lines of $\mathcal R$ at grid points. Hence, if $Y\subset L_+$ is
	such that $G_n\cup Y$ is collinearly complete with respect to
	$\mathcal R\cup\{L_+\}$, then $X_r\in Y$ for every $r$. Therefore
	$X_n\subseteq Y$, and $X_n$ is the unique smallest completion.
	
	The construction is independent of $c$, since the factor $c$ cancels
	when the $V_-$ components are eliminated in
	Lemma~\ref{lem:spectral-transversals}. For the opposite ruling the
	same argument applies after interchanging the orthogonal summands
	$V_+$ and $V_-$. The corresponding spectral line is then $L_-$, and
	the analogous $n$-point orbit on $L_-$ is again the unique smallest
	collinear completion.
\end{proof}

\begin{remark}
	The notion of collinear completeness was introduced in
	\cite{chiantini2023combinatorics} in the study of geproci sets arising
	from configurations of skew lines in projective three-space.
	The configuration $G_n\cup X_n$ is an instance of the
	\emph{standard construction} studied in
	\cite{ChiantiniEtAl2022,chiantini2023combinatorics, Favacchio2025PGL2}.
	In the present cyclic setting, the additional line is the spectral
	line $L_+$, and the points of $X_n$ are exactly those forced by the
	transversals of the ruling lines.
\end{remark}

\subsection{The completed dual-grid configuration}
\label{subsec:completed-dual-grid}

We now apply Theorem~\ref{thm:collinear-completion} to the two grids
of $Z_{n,n}=A\sqcup B$. For either grid,
the two rulings determine the same two $n$-point subsets of the
spectral lines. We denote them by
\[
X_n\subset L_+,
\qquad
X'_n\subset L_-.
\]
Thus the pair of dual grids canonically determines the two completion
orbits $X_n$ and $X'_n$. Since $L_+$ and $L_-$ are mutually polar,
every ray of $X_n$ is orthogonal to every ray of $X'_n$.

We define the \emph{completed dual-grid configuration} by
\[
\widehat Z_n
=
Z_{n,n}\sqcup X_n\sqcup X'_n.
\]
It consists of $2n(n+1)$ rays. With the notation of
Section~\ref{sec:grids}, the completion orbits may be indexed as
\[
P_r=
\left[u\left(\frac{\pi}{2}+\frac{2\pi r}{n}\right):0\right],
\qquad
P'_s=
\left[0:u\left(\frac{\pi}{2}+\frac{2\pi s}{n}\right)\right],
\qquad r,s\in\ZZ_n.
\]

Although the completion is defined by projective incidence, the
added rays reveal a second canonical family of orthogonal contexts.

\begin{theorem}[Orthogonal contexts from the completion]
	\label{thm:completion-contexts}
	Let $n\geq3$ be odd, and let
	$\widehat Z_n=Z_{n,n}\sqcup X_n\sqcup X'_n$ be the canonical
	collinear completion of $Z_{n,n}$. Write
	\[
	\alpha=\frac{2\pi a}{n},
	\qquad
	\beta=\frac{2\pi b}{n},
	\qquad
	a,b\in\ZZ_n^\times.
	\]
	For every $(r,s)\in\ZZ_n^2$, there is a unique
	$(i,j)\in\ZZ_n^2$ such that
	\[
	D_{r,s}=\{P_r,P'_s,A_{ij},B_{ij}\}
	\]
	is an orthogonal context. Explicitly,
	\[
	i=(2a)^{-1}(r+s),
	\qquad
	j=(2b)^{-1}(s-r)
	\quad\text{in }\ZZ_n.
	\]
	Consequently, the completion determines a canonical family of
	$n^2$ additional orthogonal contexts. Each grid ray $A_{ij}$ and
	$B_{ij}$ occurs in exactly one of them, while each completion ray
	$X_r$ and $X'_s$ occurs in exactly $n$.
\end{theorem}

\begin{proof}
	From the coordinates of Section~\ref{sec:grids}, the conditions
	$P_r\perp A_{ij}$ and $P'_s\perp A_{ij}$ are equivalent,
	respectively, to
	\[
	ai-bj\equiv r\pmod n,
	\qquad
	ai+bj\equiv s\pmod n.
	\]
	The same congruences characterize the orthogonality of $B_{ij}$ to
	$P_r$ and $P'_s$, since the corresponding planar components have
	the same angular arguments.
	
	Since $n$ is odd and $a,b$ are units modulo $n$, both $2a$ and $2b$
	are invertible in $\ZZ_n$. Adding and subtracting the congruences
	gives the unique solution
	\[
	i=(2a)^{-1}(r+s),
	\qquad
	j=(2b)^{-1}(s-r).
	\]
	For this pair, $A_{ij}\perp B_{ij}$ directly from their defining
	coordinates, while $P_r\perp P'_s$ because $L_+$ and $L_-$ are
	mutually polar. Hence the four rays form an orthogonal context.
	
	As $(r,s)$ ranges over $\ZZ_n^2$, the displayed linear
	transformation between $(r,s)$ and $(i,j)$ is bijective. Thus every
	grid ray occurs in exactly one of these contexts, while fixing $r$
	and varying $s$ gives the $n$ contexts containing $P_r$, and
	similarly for $P'_s$.
\end{proof}

The canonical families $\{C_{ij}\}$ and $\{D_{r,s}\}$ need not exhaust the full orthogonality hypergraph of $\widehat Z_n$, which may depend on the Euclidean realization. Examples and computational data illustrating this dependence are recorded in
Section~\ref{app:completed-orthogonality} of the supplementary
material.

\begin{proposition}
	\label{prop:completed-incidence}
	Let
	$\mathcal C=\{C_{ij}\}_{(i,j)\in\ZZ_n^2}$ be the canonical contexts
	of $Z_{n,n}$ and let
	$\mathcal D=\{D_{r,s}\}_{(r,s)\in\ZZ_n^2}$ be the additional
	contexts of Theorem~\ref{thm:completion-contexts}. Then
	$\widehat Z_n$ consists of $2n(n+1)$ rays and
	$\mathcal C\sqcup\mathcal D$ consists of $2n^2$ distinct orthogonal
	contexts. With respect to this canonical family, every grid ray
	$A_{ij}$ and $B_{ij}$ occurs in exactly three contexts, whereas every
	completion ray $X_r$ and $X'_s$ occurs in exactly $n$ contexts.
\end{proposition}

\begin{proof}
	The ray count follows from
	$|Z_{n,n}|=2n^2$ and $|X_n|=|X'_n|=n$.
	The family $\mathcal C$ consists of the $n^2$ contexts of the
	dual-grid construction, in which each grid ray occurs exactly twice.
	By Theorem~\ref{thm:completion-contexts}, the family $\mathcal D$
	also consists of $n^2$ contexts, each grid ray occurs in exactly one
	of them, and each completion ray occurs in exactly $n$.
	
	The two families are disjoint, since every context in $\mathcal C$
	is contained in $Z_{n,n}$, whereas every context in $\mathcal D$
	contains one ray of $X_n$ and one ray of $X'_n$. The asserted
	multiplicities follow.
\end{proof}

\begin{remark}
	\label{rem:incidence-to-orthogonality}
	The passage from $Z_{n,n}$ to $\widehat Z_n$ is defined by projective
	incidence, but it reveals new metric structure. The $2n$ rays forced
	by collinear completeness organize themselves into $n^2$ new
	orthogonal contexts with the original dual grids. Thus the two
	canonical families have different geometric origins: the contexts
	$C_{ij}$ arise from polar pairs of secants of the dual grids, whereas
	the contexts $D_{r,s}$ emerge only after the collinear completion and
	involve the two spectral lines.
\end{remark}

\begin{remark}[Collinear versus orthogonal completion]
	\label{rem:collinear-orthogonal-completion}
	The completion considered here is projective rather than orthogonal.
	The spectral lines are determined by the Euclidean pair $(Q,h)$, but
	once the relevant line arrangement has been identified, the added
	points are forced by collinear incidence. This need not coincide with
	closure under orthogonality. For example, in the $H_4$-compatible
	realization of $Z_{5,5}$, the original $50$ rays are already
	orthogonally closed, whereas the collinear completion adds ten rays
	and yields the projectivized $H_4$ configuration; see
	Section~\ref{app:Z55-closure} of the supplementary material for an
	exact verification.
\end{remark}

For $n=3$ and $n=5$, the completed configurations acquire exceptional
symmetry. We identify these two cases in the next section.

\section{Exceptional completions: $F_4$ and $H_4$}
\label{sec:exceptional}

The first two members of the diagonal family recover two exceptional root configurations after collinear completion.  For $n=3$, the canonical collinear
completion of the Cabello configuration restores the projectivized
$F_4$ root configuration, equivalently the Peres $24$-ray system.
For $n=5$, a distinguished admissible Euclidean realization of the
dual-grid construction completes to the projectivized $H_4$ root
configuration.

These identifications are consequences of the general construction
of Section~\ref{sec:completion}: the root configurations are not used
to define the completion, but emerge as exceptional members of the
family $\{\widehat Z_n\}$.

\subsection{The $F_4$ completion}
\label{subsec:F4}

\begin{theorem}
	\label{thm:F4-completion}
	The completed dual-grid configuration $\widehat Z_3$ is orthogonally
	equivalent to the projectivized $F_4$ root configuration, or
	equivalently to the Peres $24$-ray configuration.
\end{theorem}

\begin{proof}
	By Theorem~\ref{thm:cabello-identification}, the two grids of
	$Z_{3,3}$ are orthogonally equivalent to the complements
	$D\setminus L$ and $D'\setminus L'$ of a pair of mutually polar
	trisecants in the two projectivized $D_4$ components of the Peres
	configuration.
	
	By Theorem~\ref{thm:collinear-completion}, the three points on each
	trisecant are precisely the canonical collinear completion of the
	corresponding $(3,3)$-grid. Adjoining the two completion orbits
	therefore restores $D\sqcup D'$, which is the projectivized $F_4$
	root configuration.
\end{proof}

Thus the passage from the Cabello $18$--$9$ configuration to the
Peres $24$-ray configuration is intrinsic to the grid geometry: the
six additional rays are forced by the collinear completion of the two
$(3,3)$-grid components.

\begin{remark}
	\label{rem:sixteen-completions}
	The choice of the pair of trisecants is not essential. Each
	projectivized $D_4$ component contains $16$ trisecants, paired across
	the two components by Euclidean polarity. The corresponding
	trisecant deletions therefore give equivalent realizations of the
	same grid-completion picture. Explicit coordinates and verification
	for one such pair are recorded in 
	Section~\ref{app:cabello} of the Supplementary Material.
\end{remark}

\subsection{The $H_4$ completion}
\label{subsec:H4}

For $n=5$, the completed configuration has
$2\cdot5^2+2\cdot5=60$ rays, the same number as the projectivized
$H_4$ root configuration. For a distinguished admissible Euclidean
realization, the agreement extends to an orthogonal equivalence.

Let $\phi=(1+\sqrt5)/2$. Consider the realization determined by
$a=1$ and $b=2$. Then
\[
\alpha=\frac{2\pi}{5},\qquad
\beta=\frac{4\pi}{5},\qquad
c^2=\phi^{-2}.
\]

\begin{theorem}
	\label{thm:H4-completion}
	For the admissible realization of $Z_{5,5}$ determined by
	$a=1$ and $b=2$, the canonical collinear completion
	$\widehat Z_5$ is orthogonally equivalent to the projectivized
	$H_4$ root configuration.
\end{theorem}

\begin{proof}
	We use the coordinates for the projectivized $H_4$ configuration
	from~\cite{harbourne2021unexpected}, with the numbering adopted
	by Wi\'sniewska and Zi\k{e}ba~\cite{WisniewskaZiebaH4}. Its $60$ points split into
	two sets of $30$ points, each consisting of a $(5,5)$-grid together
	with five collinear points.
	
	Let $Q_1$ and $Q_2$ be the quadrics supporting the two grids. In
	the coordinates of \cite{WisniewskaZiebaH4}, the symmetric matrix
	of $Q_1$ is
	\[
	M_1=
	\begin{pmatrix}
		0&1&0&0\\
		1&-1&0&0\\
		0&0&\phi-1&0\\
		0&0&0&-\phi
	\end{pmatrix},
	\]
	and the matrix of $Q_2$ is $M_1^{-1}$. Thus $Q_1$ and $Q_2$ are
	dual with respect to Euclidean polarity.
	
	The eigenvalues of $M_1$ are $\phi^{-1}$ and $-\phi$, each with
	multiplicity two. After rescaling the equation and applying an
	orthogonal change of coordinates, $Q_1$ is therefore of the form
	$Q_c$ with $c^2=\phi^{-2}$, exactly the spectral parameter of the
	dual-grid realization above. The corresponding eigenspaces are
	\[
	E_{\phi^{-1}}
	=\langle(1,\phi-1,0,0),(0,0,1,0)\rangle,
	\qquad
	E_{-\phi}
	=\langle(1,-\phi,0,0),(0,0,0,1)\rangle.
	\]
	In the notation of \cite{WisniewskaZiebaH4}, their projectivizations
	are the two five-point lines $\ell_{17}$ and $\ell_{24}$,
	respectively. Thus these are precisely the spectral lines of $Q_1$.
	
	Set 
	$d=\sqrt{1+\phi^{-2}}=\sqrt{3-\phi}$ 
	and consider
	\[
	R=
	\begin{pmatrix}
		d^{-1}&0&-\phi^{-1}d^{-1}&0\\
		\phi^{-1}d^{-1}&0&d^{-1}&0\\
		0&1&0&0\\
		0&0&0&-1
	\end{pmatrix}.
	\]
	Its columns form an orthonormal basis, so $R\in O(4)$.
	
	Let $G_1\subset Q_1$ and $G_2\subset Q_2$ denote the two
	$25$-point grids in the projectivized $H_4$ configuration. Exact
	calculation, up to nonzero scalar multiples, gives
	\[
	R(\{A_{ij}\})=G_1,\quad
	R(\{B_{ij}\})=G_2,
	\quad\text{and}\quad
	R(X_5)=\ell_{17}\cap H_4^{\mathrm{proj}},
	\quad
	R(X'_5)=\ell_{24}\cap H_4^{\mathrm{proj}}.
	\]
	The complete ray correspondence and its exact verification are
	recorded in Section~\ref{app:H4-verification} of the  Supplementary Material.
	
	Hence $R$ maps the $50$ rays of $Z_{5,5}$ onto the two grids and
	the ten completion rays onto the remaining ten points. Therefore
	$R(\widehat Z_5)=H_4^{\mathrm{proj}},$
	which proves the claim.
\end{proof}

\begin{remark}
	\label{rem:H4-parameters}
	The identification with $H_4$ depends on the Euclidean realization
	of the dual-grid configuration, not only on its projective incidence
	structure. Different admissible choices of the cyclic steps may give
	the same dual-grid incidence pattern while producing different full
	orthogonality hypergraphs; see
	Section~\ref{app:completed-orthogonality} of the supplementary
	material.
	
	The spectral value $c^2=\phi^{-2}$ selects the realization
	$a=1$, $b=2$ as a natural candidate for the $H_4$ identification,
	but spectral agreement alone does not identify the two cyclic
	$(5,5)$-grids. The orthogonal equivalence is established by the
	explicit transformation $R$ above; Section~\ref{app:H4-verification} of the Supplementary Material
	records an exact verification of the complete ray correspondence.
\end{remark}

The two exceptional cases therefore fit the same pattern:
\[
Z_{3,3}\subset\widehat Z_3\simeq_{\mathrm{orth}} F_4^{\mathrm{proj}},
\qquad
Z_{5,5}\subset\widehat Z_5\simeq_{\mathrm{orth}} H_4^{\mathrm{proj}}.
\]
The root configurations are not inputs to the construction; rather,
they emerge from the canonical completion of the first two members of
the diagonal family.

\section{Discussion and further directions}
\label{sec:discussion}

The constructions developed in this paper bring together three
structures that are often considered separately: projective
incidence, Euclidean orthogonality, and Kochen--Specker
contextuality. The dual-grid construction shows that, in dimension
four, a substantial part of the orthogonality structure of a
Kochen--Specker configuration may arise from comparatively simple
projective geometry on pairs of ruled quadrics.

The purpose of this final section is to summarize the interaction
between these structures, to clarify which features are canonical
and which depend on the Euclidean realization, and to indicate some
questions suggested by the dual-grid viewpoint.

\subsection{Incidence, polarity, and contextuality}

The first family is built from two finite grids on a pair of smooth
quadrics dual with respect to the Euclidean polarity. Its basic
mechanism may be summarized schematically as
\[
\text{dual quadrics}
\longrightarrow
\text{polar secants}
\longrightarrow
\mathcal C.
\]
The rulings of the quadrics provide the projective incidence
structure, while the polarity converts suitable pairs of secants into
orthogonality relations. For odd $p,q\geq3$, the resulting
configuration $Z_{p,q}$ consists of $2pq$ rays and carries $pq$
distinguished contexts, with every ray occurring exactly twice.
The parity obstruction is therefore already present at the level of
the dual grids. The completed diagonal case carries a second mechanism:
\[
\text{grid incidence}
\longrightarrow
\text{collinear completion}
\longrightarrow
\text{spectral polarity}
\longrightarrow
\mathcal D.
\]
For $Z_{n,n}$, collinear completion forces the addition of $n$ points
on each of two distinguished lines. The resulting
configuration
\[
Z_{n,n}\subset\widehat Z_n,
\qquad
|\widehat Z_n|=2n(n+1),
\]
has the additional Euclidean feature that the two completion lines
are precisely the spectral lines of the supporting quadrics and are
mutually polar. This interaction between the projective completion
and the Euclidean polarity produces the second canonical family
$\mathcal D$ of $n^2$ orthogonal contexts.

The two mechanisms should be kept conceptually distinct. The
Kochen--Specker obstruction belongs already to $Z_{n,n}$, before the
additional $2n$ rays are introduced. The role of the completion is
different: it reveals further projective incidence and, through its
interaction with the Euclidean polarity, further orthogonality.

The exceptional configurations $F_4$ and $H_4$ fit naturally into
this picture. The identifications
\[
\widehat Z_3\simeq_{\mathrm{orth}} F_4^{\mathrm{proj}},
\qquad
\widehat Z_5\simeq_{\mathrm{orth}} H_4^{\mathrm{proj}}
\]
show that the first two exceptional completions coincide, up to orthogonal equivalence, with the projectivizations of the finite root systems $F_4$ and $H_4$, respectively. The root systems are therefore not needed to define the construction: rather, they emerge
as exceptional members of an infinite geometric family.

\subsection{Canonical contexts and the full orthogonality hypergraph}

The construction singles out two canonical families of orthogonal
contexts. The first family $\mathcal C$ is present for every
admissible dual-grid configuration and is determined by the polar
secants of the two grids. In the completed diagonal case, the
interaction between the completion lines and Euclidean polarity
produces the second family $\mathcal D$.

These canonical contexts need not exhaust the orthogonality
relations of a particular Euclidean realization. It is therefore
useful to distinguish the incidence structure produced by the
construction from the \emph{full orthogonality hypergraph}, whose
vertices are the rays of the configuration and whose hyperedges are
all orthogonal bases contained in it.

This distinction is already visible in the smallest cases. Exact
enumeration shows that both the completed $n=3$ configuration,
identified with $F_4^{\mathrm{proj}}$, and the $H_4$-compatible
$n=5$ realization contain orthogonal bases beyond those distinguished
by the two canonical families. Exact enumerations are recorded in
Section~\ref{app:completed-orthogonality} of the Supplementary
Material.

More generally, the full orthogonality hypergraph is a metric
feature of the chosen Euclidean realization and need not be
determined by the underlying dual-grid incidence structure. Exact
enumerations in the first cases show that different admissible
choices of the cyclic steps may produce different numbers of
orthogonal bases; see
Section~\ref{app:completed-orthogonality} of the Supplementary
Material.

\begin{remark}
	\label{rem:full-orthogonality}
	The distinction between the canonical context structure and the full
	orthogonality hypergraph is not merely formal. For fixed $p$ and $q$,
	the canonical families supplied by the construction have the same
	cardinalities and incidence pattern for every admissible choice of
	cyclic steps, whereas the full set of orthogonal bases may depend on
	the resulting Euclidean realization. Thus the projective dual-grid
	geometry determines a canonical orthogonality structure, while
	special metric realizations may exhibit additional orthogonality.
	
	This does not affect the parity proof. The canonical family
	$\mathcal C$ already forms a parity-proof subhypergraph: it contains
	an odd number of contexts and every ray occurs in it exactly twice.
	Additional orthogonal bases enlarge the ambient orthogonality
	hypergraph without invalidating this obstruction.
\end{remark}

From the physical point of view, the role of these additional
contexts is less clear. They provide additional compatible
measurement bases on the same set of rays, but need not participate
in the parity proof used in this paper.  

\begin{question}
	For a fixed dual-grid incidence structure, what is the physical
	significance of varying the Euclidean realization and hence the full
	orthogonality hypergraph? In particular, does the presence of
	additional orthogonal contexts lead to stronger or more robust
	certificates of contextuality?
\end{question}

\subsection{The converse at the minimum}

The smallest member of the construction is particularly suggestive.
For $p=q=3$, the dual-grid configuration $Z_{3,3}$ has $18$ rays
and $9$ canonical contexts and is orthogonally equivalent to the
Cabello configuration. This is the minimum possible number of rays
for a Kochen--Specker configuration in real dimension four.

From the present point of view, the Cabello configuration carries
considerably more geometry than is required by its usual
description as an $18$-ray parity proof: its rays split into two
$(3,3)$-grids on a pair of smooth Euclidean-dual quadrics, and its
nine contexts arise from pairs of polar secants.

The metric freedom of the dual-grid
	construction collapses. Indeed, the only units modulo $3$ are
	$1$ and $2$, and the admissible choices of cyclic steps yield
	$c^2=2$ or $c^2=1/2$.  Thus, up to orthogonal
	equivalence, the dual-grid construction has a unique Euclidean
	realization in the $(3,3)$ case.

This raises a natural rigidity question. Is this additional
projective geometry a special feature of the known minimal
configuration, or is it forced by minimality itself?

\begin{question}
	Is every $18$-ray Kochen--Specker configuration in
	$\PP^3(\RR)$ projectively equivalent to a union of two
	$(3,3)$-grids on smooth quadrics? If so, can the quadrics always be
	chosen to be dual with respect to a Euclidean polarity for which
	the nine contexts arise from the polar-secant construction?
\end{question}
The first part is purely projective, whereas the second asks whether
the Euclidean polarity responsible for contextuality is itself
forced by the minimal orthogonality structure.

A positive answer would give a geometric characterization of
minimal Kochen--Specker configurations in real dimension four.
Even a weaker statement---for example, that every minimal
configuration admits some distinguished pair of ruled quadrics or
some comparable projective decomposition---would indicate that the
geometry exhibited here reflects rigidity at the minimum rather
than an accidental feature of one realization.

\subsection{Towards a converse}
\label{subsec:converse}

The results of this paper start from a pair of grids on
Euclidean-dual quadrics and produce Kochen--Specker configurations.
It is natural to ask to what extent the geometry can be recovered in
the opposite direction. We do not know a converse to our
construction, but the problem can be formulated intrinsically in
projective terms.

Let $Z\subset\PP^3(\RR)$ be a real Kochen--Specker configuration.
Each context is the projectivization of an orthogonal basis of
$\RR^4$, and hence determines a self-polar tetrahedron with respect
to the Euclidean polarity. In particular, every decomposition of a
context into two pairs,
\[
\{P_1,P_2,P_3,P_4\}
=
\{P_1,P_2\}\sqcup\{P_3,P_4\},
\]
gives a pair of polar lines $\langle P_1,P_2\rangle^{\perp_h}
=
\langle P_3,P_4\rangle.$

Thus polar pairs of secants are already present locally in every
four-dimensional KS configuration. What is not automatic is that
these secants organize globally into the rulings of two quadrics.

There is another elementary mechanism by which orthogonality can
force projective incidence. If three distinct points
$R_1,R_2,R_3\in Z$ are orthogonal to two distinct points
$P,Q\in Z$, then
\[
R_1,R_2,R_3\in
\PP\bigl(\langle P,Q\rangle^{\perp_h}\bigr),
\]
and hence $R_1,R_2,R_3$ are collinear. Equivalently, a $K_{3,3}$
in the orthogonality graph of $Z$ determines a pair of polar
trisecants. More generally, sufficiently large families of common
orthogonals produce multisecants without any collinearity assumption
being imposed in advance.

This identifies a first obstruction to a converse. The
Kochen--Specker condition is a global non-colorability condition on
the hypergraph of contexts, whereas the existence of grids requires
specific incidence relations among the corresponding projective
points. In particular, the KS condition alone does not visibly force
a $K_{3,3}$ in the orthogonality graph. Establishing such a
statement under additional hypotheses, or finding an obstruction to
it, would already provide a first bridge from contextuality to
projective incidence.

Even the existence of polar trisecants would not yet recover the
geometry considered here. A $(p,q)$-grid requires two families of
multisecants forming the two rulings of a smooth quadric. For
instance, three skew lines together with three common transversals
determine a smooth quadric and a $(3,3)$-subgrid. Under Euclidean
polarity the corresponding configuration is carried to the dual
quadric. Thus, once a sufficiently rich system of polar
multisecants is forced, the passage from orthogonality to a pair of
dual ruled quadrics becomes projectively natural. This suggests separating a possible converse into three steps:
\[
\text{KS non-colorability}
\ \Longrightarrow\
\text{polar multisecants}
\ \Longrightarrow\
\text{two ruled quadrics}
\ \Longrightarrow\
\text{dual grids}.
\]
Only the local ingredients of this chain are automatic; the global
implications remain unclear.

For parity configurations there is additional combinatorial
structure. If every ray occurs exactly twice, the contexts may be
regarded as the vertices of a $4$-regular graph and the rays as its
edges. Since every finite $4$-regular graph admits a $2$-factorization,
choosing one separates the rays
into two classes such that every context contains two rays of each
class. Consequently every context becomes a pair of polar secants,
one in each class. Such a factorization is combinatorial rather than
geometric and need not be unique; by itself it does not force either
class to lie on a quadric, nor does it force the resulting secants to
form rulings. This appears to isolate precisely the additional
geometric rigidity present in the configurations constructed here.

It is therefore natural to ask whether dual-grid geometry
characterizes a distinguished class of real four-dimensional
contextual configurations.

\begin{question}
	Which real Kochen--Specker configurations in $\PP^3(\RR)$ contain
	finite grids on a pair of Euclidean-dual smooth quadrics? Is there a
	natural combinatorial or orthogonality-theoretic criterion
	characterizing this class?
\end{question}

\begin{question}
	Does every real parity Kochen--Specker configuration in
	$\PP^3(\RR)$ contain a subconfiguration orthogonally equivalent to
	some $Z_{p,q}$? More generally, under what hypotheses does a parity
	configuration contain a dual-grid subconfiguration?
\end{question}

\subsection{A quaternionic viewpoint}
\label{subsec:quaternionic-viewpoint}

We now give a complementary interpretation of the normal form
used throughout the paper. Fix an isometric identification $\Psi:\RR^4\to\mathbb H$. After a suitable ordering of coordinates,
the grid vectors of Section~\ref{sec:grids} may be written as
\[
\Psi(a_{ij})
=
e^{\mathbf i i\alpha}(-c+\mathbf j)e^{\mathbf i j\beta},
\qquad
\Psi(b_{ij})
=
e^{\mathbf i i\alpha}(1+c\mathbf j)e^{\mathbf i j\beta}.
\]
Thus the two grids are finite orbits of the maximal torus
\[
U(1)\times U(1)
\subset Sp(1)\times Sp(1)
\simeq \operatorname{Spin}(4),
\]
acting on $\mathbb H$ by left and right multiplication. In this
description the two angular parameters encode the two circle
actions, whereas $c$ records the relative sizes of the components
of the seed quaternion in the two invariant real planes. This
interpretation is available for the whole dual-grid family and does
not depend on the exceptional behavior of the first diagonal cases.

For $n=3$ and $n=5$, however, the completed configurations meet
finite quaternionic geometry in a particularly rigid way. The
projectivized $F_4$ configuration may be represented using the
binary octahedral group $2O\subset Sp(1)$. In the realization of
Theorem~\ref{thm:F4-completion}, the roots above each completion
trisecant form a coset of a cyclic subgroup of order $6$; after
projectivization, the three completion rays form a cyclic orbit of
order $3$.

Similarly, the roots of $H_4$ may be identified with the binary
icosahedral group $2I\subset Sp(1)$. Let $L$ be one of the
five-point completion lines occurring in
Theorem~\ref{thm:H4-completion}, and choose a root $q_0\in2I$
representing one of its points. Right multiplication by $q_0^{-1}$
sends the two-dimensional real subspace above $L$ to a plane
\[
W=\langle1,u\rangle_{\RR}\subset\mathbb H
\]
for some purely imaginary unit quaternion $u$. Since
$W\cap Sp(1)$ is a circle subgroup, the finite group $2I\cap W$ is
cyclic. The five projective points of $L$ lift to ten antipodal
roots.This suggests viewing the roots above $L$ as a coset of a cyclic subgroup
\[
\mathbb Z_{10}\subset 2I,
\]
or, projectively, the five completion rays as a cyclic orbit of order $5$.
It follows that the roots above $L$ form a coset of a cyclic
subgroup of order $10$, or, projectively, a cyclic orbit of order
$5$. Since the two completion lines are mutually polar, their underlying
two-dimensional vector subspaces are orthogonal. From this viewpoint,
the twenty roots above them naturally form an $H_2\oplus H_2$
subsystem of $H_4$.

The two exceptional completions thus exhibit the parallel
quaternionic pattern
\[
\mathbb Z_6\subset2O
\quad\text{for }n=3,
\qquad
\mathbb Z_{10}\subset2I
\quad\text{for }n=5,
\]
or, after projectivization, cyclic completion orbits of orders $3$
and $5$.

The first case beyond these two exceptional completions is $n=7$.
Here the quaternionic torus description remains unchanged, and for
each admissible realization the canonical collinear completion still
produces two cyclic $7$-point orbits on the spectral lines. What
fails is the further identification with one of the exceptional
binary polyhedral groups. An identification of the same kind would would require an
element of order $14$, whereas no such element occurs in the binary
tetrahedral, octahedral, or icosahedral groups $2T$, $2O$, or $2I$.
Thus the toric dual-grid construction and its spectral completion
persist for $n=7$, while the exceptional finite quaternionic
realizations underlying the $F_4$ and $H_4$ cases do not.

\subsection{Further questions}
\label{subsec:further-questions}

The first case beyond the exceptional completions provides a natural
testing ground for separating the structures encountered in this
paper. For $n=3$ and $n=5$, the completed configurations are
identified with the projectivized $F_4$ and $H_4$ root
configurations and are geproci. By contrast, the realization
$\widehat Z_7(1,2)$, which maximizes the number of orthogonal bases
among the admissible cyclic realizations considered here, is not
geproci; an exact verification is recorded in
Section~\ref{app:Z77-geproci} of the Supplementary Material.

Thus the dual-grid construction, the spectral lines, the canonical
collinear completion, and a rich orthogonality structure all persist
for $n=7$, while the exceptional Coxeter interpretation and the
geproci property do not. This suggests that the additional structures
coinciding in the first two cases should be regarded, a priori, as
distinct phenomena.

\begin{question}
	For which odd $n$ and admissible pairs $(a,b)$ is the completed
	configuration $\widehat Z_n(a,b)$ geproci? Is the geproci property
	constant among the different Euclidean realizations for fixed $n$?
\end{question}

The possible relation between geproci geometry and contextuality
extends beyond the present family. In particular, the Penrose
configuration \cite{ZP} provides another natural testing ground. It would be
interesting to determine whether the occurrence of geproci
configurations in examples arising from contextuality reflects a
genuine geometric mechanism, or whether the two properties are
essentially independent.

A second direction concerns the additional orthogonality that may
appear in special Euclidean realizations. The computations in the
first cases suggest that the number of orthogonal bases can vary with
the admissible cyclic steps, even though the canonical context
structure remains unchanged.

\begin{question}
	For fixed odd $n$, how does the full orthogonality hypergraph of
	$\widehat Z_n(a,b)$ depend on the admissible cyclic steps $(a,b)$?
	Which choices maximize the number of orthogonal bases, and do the
	maximizing realizations admit an intrinsic geometric
	characterization? Is such maximality related to orthogonal closure?
\end{question}

This problem is naturally connected with symmetry. Additional
orthogonal bases may reflect an enhancement of the symmetry of the
Euclidean realization rather than merely an accidental increase in
the number of contexts. It would therefore be useful to compare the
projective, orthogonal, and hypergraph automorphism groups of the
completed configurations and to understand how these groups vary
with the realization. In particular, one may ask whether the
realizations maximizing the number of orthogonal bases are also
distinguished by an enhanced automorphism group.

The cyclic symmetry built into the construction is another feature
that may not be essential. The parity argument ultimately depends on
the incidence pattern of the distinguished contexts and on the
polar-secant mechanism, rather than on cyclic symmetry itself.

\begin{question}
	Can the dual-grid construction be extended to non-cyclic finite
	subsets of pairs of Euclidean-dual quadrics? More generally, which
	finite grids admit families of polar secants producing regular
	orthogonality hypergraphs and parity proofs?
\end{question}

Finally, the appearance of split quadrics suggests a possible
connection with tensor-product geometry. Projectively, a smooth split
quadric is a Segre surface
\[
\PP^1\times\PP^1\hookrightarrow\PP^3,
\]
whose two rulings correspond to the two factors. Thus the supporting
quadrics naturally carry a projective tensor-product structure. In
the present construction this structure need not be compatible with
the Euclidean inner product defining quantum orthogonality, and we
have therefore not assigned it an operational interpretation.

\begin{question}
	Does the interaction between the Segre structure associated with the
	supporting quadrics and the Euclidean polarity defining orthogonality
	admit a natural quantum-mechanical interpretation?
\end{question}

Taken together, these questions suggest that finite Kochen--Specker
configurations in dimension four may carry substantially more
projective geometry than is visible from their orthogonality
hypergraphs alone. The dual-grid viewpoint provides one systematic
way of exposing this geometry, with polarity mediating between
projective incidence and quantum orthogonality.

\section*{Acknowledgments}
The author is a member of GNSAGA-INdAM.   The work of the author was supported by the funding PREMIO\_SINGOLI\_RIC\_[2025] from the Department of Engineering, University of Palermo. The author used Chat-GPT (OpenAI) for assistance with improving the readability of the manuscript and polishing
the accompanying code. The author assumes responsibility for all content.

\appendix

\section*{Supplementary Material}

\section{Coordinates and verification for the Peres configuration}
\label{app:cabello}

We record here the explicit coordinates and computations used in
Section~\ref{sec:cabello} to identify the dual-grid configuration
$Z_{3,3}$ with the Cabello $18$--$9$ configuration inside the Peres
$24$-ray configuration. The calculations also provide a direct
coordinate verification of the $(3,3)$-grid structure appearing in
Lemma~\ref{lem:D4-minus-trisecant} and of the orthogonal equivalence
in Theorem~\ref{thm:cabello-identification}.

We use the following realization of the Peres $24$-ray configuration:

\[
\begin{aligned}
P_1&=(2,0,0,0),&
P_2&=(0,2,0,0),&
P_3&=(0,0,2,0),&
P_4&=(0,0,0,2),\\
P_5&=(1,1,1,1),&
P_6&=(1,1,-1,-1),&
P_7&=(1,-1,1,-1),&
P_8&=(1,-1,-1,1),\\
P_9&=(1,1,1,-1),&
P_{10}&=(1,1,-1,1),&
P_{11}&=(1,-1,1,1),&
P_{12}&=(1,-1,-1,-1),
\end{aligned}
\]
and
\[
\begin{aligned}
P_{13}&=(1,1,0,0),&
P_{14}&=(1,-1,0,0),&
P_{15}&=(0,0,1,1),&
P_{16}&=(0,0,1,-1),\\
P_{17}&=(0,1,0,1),&
P_{18}&=(0,1,0,-1),&
P_{19}&=(1,0,1,0),&
P_{20}&=(1,0,-1,0),\\
P_{21}&=(1,0,0,1),&
P_{22}&=(1,0,0,-1),&
P_{23}&=(0,1,1,0),&
P_{24}&=(0,1,-1,0).
\end{aligned}
\]
Set
\[
\mathcal D=\{P_1,\ldots,P_{12}\},
\qquad
\mathcal D'=\{P_{13},\ldots,P_{24}\}.
\]
These are the two projectivized $D_4$ components of the
Peres--$F_4$ configuration.

Each component contains $16$ trisecants. The Euclidean polarity pairs
them into $16$ mutually orthogonal pairs, in the sense that every
point of one trisecant is orthogonal to every point of the other.
For the identification used in Section~\ref{sec:cabello}, we choose
\[
L=\{P_1,P_5,P_{12}\},
\qquad
L'=\{P_{16},P_{18},P_{24}\}.
\]
Their complements are the two nine-point sets
\[
\mathcal D\setminus L
=
\{P_2,P_3,P_4,P_6,P_7,P_8,P_9,P_{10},P_{11}\}
\]
and
\[
\mathcal D'\setminus L'
=
\{P_{13},P_{14},P_{15},P_{17},P_{19},P_{20},
  P_{21},P_{22},P_{23}\}.
\]

The six grid lines of $\mathcal D\setminus L$ split into the two
rulings
\[
\begin{aligned}
&(P_2P_7P_9),\quad
(P_3P_8P_{11}),\quad
(P_4P_6P_{10}),\\
&(P_2P_8P_{10}),\quad
(P_3P_6P_9),\quad
(P_4P_7P_{11}),
\end{aligned}
\]
while those of $\mathcal D'\setminus L'$ split as
\[
\begin{aligned}
&(P_{13}P_{17}P_{22}),\quad
(P_{14}P_{19}P_{23}),\quad
(P_{15}P_{20}P_{21}),\\
&(P_{13}P_{20}P_{23}),\quad
(P_{14}P_{17}P_{21}),\quad
(P_{15}P_{19}P_{22}).
\end{aligned}
\]
Thus both complements are $(3,3)$-grids, as asserted in
Lemma~\ref{lem:D4-minus-trisecant}.

For $Z_{3,3}$ we take $\alpha=\beta=2\pi/3$ and $c=\sqrt2$.
The orthogonal matrix
\[
R=
\begin{pmatrix}
0&1&0&0\\
\frac1{\sqrt3}&0&-\sqrt{\frac23}&0\\
\frac1{\sqrt3}&0&\frac1{\sqrt6}&-\frac1{\sqrt2}\\
\frac1{\sqrt3}&0&\frac1{\sqrt6}&\frac1{\sqrt2}
\end{pmatrix}
\]
maps the two grids $A$ and $B$ onto
$\mathcal D\setminus L$ and $\mathcal D'\setminus L'$, respectively.
More explicitly,
\[
\begin{array}{c|ccccccccc}
A_{ij}
&A_{00}&A_{01}&A_{02}&A_{10}&A_{11}&A_{12}&A_{20}&A_{21}&A_{22}\\
\hline
R(A_{ij})
&P_2&P_{10}&P_8&P_7&P_4&P_{11}&P_9&P_6&P_3
\end{array}
\]
and
\[
\begin{array}{c|ccccccccc}
B_{ij}
&B_{00}&B_{01}&B_{02}&B_{10}&B_{11}&B_{12}&B_{20}&B_{21}&B_{22}\\
\hline
R(B_{ij})
&P_{15}&P_{19}&P_{22}&P_{20}&P_{23}&P_{13}
&P_{21}&P_{14}&P_{17}.
\end{array}
\]

Under the same transformation, the nine canonical contexts are
mapped as follows:
\[
\begin{array}{c|c@{\qquad}c|c@{\qquad}c|c}
C_{00}&\{P_2,P_3,P_{21},P_{22}\}&
C_{01}&\{P_9,P_{10},P_{14},P_{15}\}&
C_{02}&\{P_6,P_8,P_{17},P_{19}\}\\
C_{10}&\{P_7,P_8,P_{13},P_{15}\}&
C_{11}&\{P_2,P_4,P_{19},P_{20}\}&
C_{12}&\{P_{10},P_{11},P_{22},P_{23}\}\\
C_{20}&\{P_9,P_{11},P_{17},P_{20}\}&
C_{21}&\{P_6,P_7,P_{21},P_{23}\}&
C_{22}&\{P_3,P_4,P_{13},P_{14}\}.
\end{array}
\]
These are exactly the nine orthogonal bases in the chosen
$18$-ray complement containing two rays from each $D_4$ component.

\subsection*{Exact verification}
For reproducibility, the SageMath \texttt{cabello\_peres.sage} code verifies
symbolically all of the preceding statements used in
Section~\ref{sec:cabello}. It checks that the $18$ dual-grid rays are
distinct and that the nine $C_{ij}$ are orthogonal contexts;
enumerates the $16$ trisecants in each $D_4$ component and the $16$
mutually orthogonal pairs; verifies that every corresponding
nine-point complement is a $(3,3)$-grid; checks that $R$ is
orthogonal and realizes the displayed ray correspondence; and
finally verifies that the images of the nine $C_{ij}$ are exactly
the nine orthogonal $2+2$ bases of the chosen complement. All
computations are performed symbolically in SageMath; no
floating-point approximation, numerical tolerance, or numerical
recognition is used.

\subsection*{Output}In the code we call a pair $(L,M)$ a Hexagon when $L$ is a
trisecant of $\mathcal D$, $M$ is a trisecant of $\mathcal D'$,
and every ray of $L$ is orthogonal to every ray of $M$.

{\footnotesize
\begin{verbatim}
Verified: the dual-grid configuration contains 18 distinct rays.
Verified: all nine C(i,j) are orthogonal contexts.
Verified: the Peres realization is split into two 12-ray components.
Verified: each D4 component contains 16 trisecants.
Verified: there are 16 mutually orthogonal trisecant pairs.
Verified: every orthogonal trisecant pair has two (3,3)-grid complements.
Verified: the Hexagon used in the paper is one of the 16 grid-complement Hexagons.
Verified: the chosen 18-ray complement contains exactly nine orthogonal 2+2 bases.
Verified: R is orthogonal.
Verified: R maps A bijectively onto D \ L.
Verified: R maps B bijectively onto D' \ L'.
Verified: the nine C(i,j) are exactly the nine orthogonal 2+2 bases of the chosen
          Peres complement.

Chosen Hexagon:
L  = (1, 5, 12)
L' = (16, 18, 24)

Rulings of D \ L:
([(2, 7, 9), (3, 8, 11), (4, 6, 10)], [(2, 8, 10), (3, 6, 9), (4, 7, 11)])

Rulings of D' \ L':
([(13, 17, 22), (14, 19, 23), (15, 20, 21)], [(13, 20, 23), (14, 17, 21), (15, 19, 22)])

Orthogonal matrix R:
[           0            1            0            0]
[ 1/3*sqrt(3)            0 -1/3*sqrt(6)            0]
[ 1/3*sqrt(3)            0  1/6*sqrt(6) -1/2*sqrt(2)]
[ 1/3*sqrt(3)            0  1/6*sqrt(6)  1/2*sqrt(2)]

A correspondence:
A(0, 0) -> P2
A(0, 1) -> P10
A(0, 2) -> P8
A(1, 0) -> P7
A(1, 1) -> P4
A(1, 2) -> P11
A(2, 0) -> P9
A(2, 1) -> P6
A(2, 2) -> P3

B correspondence:
B(0, 0) -> P15
B(0, 1) -> P19
B(0, 2) -> P22
B(1, 0) -> P20
B(1, 1) -> P23
B(1, 2) -> P13
B(2, 0) -> P21
B(2, 1) -> P14
B(2, 2) -> P17

Context correspondence:
C(0,0) -> [2, 3, 21, 22]
C(0,1) -> [9, 10, 14, 15]
C(0,2) -> [6, 8, 17, 19]
C(1,0) -> [7, 8, 13, 15]
C(1,1) -> [2, 4, 19, 20]
C(1,2) -> [10, 11, 22, 23]
C(2,0) -> [9, 11, 17, 20]
C(2,1) -> [6, 7, 21, 23]
C(2,2) -> [3, 4, 13, 14]
\end{verbatim}
}

The computation gives a direct verification of the explicit
identification used in Theorem~\ref{thm:cabello-identification}.
Each of the two $D_4$ components contains $16$ trisecants, and the
Euclidean orthogonality relation pairs them into $16$ mutually
orthogonal pairs.  It also confirms directly that the corresponding
complements are $(3,3)$-grids, in agreement with
Lemma~\ref{lem:D4-minus-trisecant}.

For the pair
$L=\{P_1,P_5,P_{12}\},
\ 
L'=\{P_{16},P_{18},P_{24}\},$
the orthogonal matrix $R$ displayed above maps the two grids
$A$ and $B$ bijectively onto
$\mathcal D\setminus L$ and $\mathcal D'\setminus L'$, respectively.
Moreover, the images of the nine contexts $C_{ij}$ are exactly the
nine orthogonal $2+2$ bases contained in this $18$-ray complement.
This gives the explicit verification of the orthogonal equivalence
used in Theorem~\ref{thm:cabello-identification}.

\section{Orthogonality data for the completed configurations}
\label{app:completed-orthogonality}

We record here the computational data concerning the full
orthogonality structure of the first three completed diagonal
configurations $\widehat Z_n$. These computations supplement the
canonical context families described in
Section~\ref{sec:completion} of the main paper.

Recall that $\widehat Z_n$ decomposes as
\[
\widehat Z_n=A\sqcup B\sqcup X_n\sqcup X'_n.
\]
The two canonical families of orthogonal contexts are
\[
\mathcal C=\{C_{ij}\}_{(i,j)\in\ZZ_n^2},
\qquad
\mathcal D=\{D_{r,s}\}_{(r,s)\in\ZZ_n^2}.
\]
The contexts in $\mathcal C$ have type $2A+2B$, while those in
$\mathcal D$ have type $A+B+X+X'$. Here the type of an orthogonal
basis records the number of its rays lying in the four components
$A$, $B$, $X_n$, and $X'_n$.

A direct enumeration of all orthogonal bases in the first three
diagonal cases gives
\[
\begin{array}{c|c|c|l}
	n
	& |\widehat Z_n|
	& \#\text{orthog. bases}
	& \text{distribution by type}
	\\
	\hline
	3
	& 24
	& 24
	& 9(2A+2B)+9(A+B+X+X')
	+3(3A+X)+3(3B+X')
	\\
	5
	& 60
	& 75
	& 50(2A+2B)+25(A+B+X+X')
	\\
	7
	& 112
	& 147
	& 98(2A+2B)+49(A+B+X+X')
\end{array}
\]

The total number of orthogonal bases depends on the Euclidean
realization, and hence on the admissible cyclic steps $(a,b)$.
For $n=5$ and $n=7$, an exhaustive enumeration of all admissible
pairs shows that only two values occur, namely $2n^2$ and $3n^2$.
The choice $(a,b)=(1,2)$ attains the maximal value $3n^2$ in both
cases. It is not the unique maximizing choice; the complete lists of
maximizing pairs are produced by the computation below.

Several features of the enumeration are worth noting. In each of
the three cases there are exactly $n^2$ bases of type
$A+B+X+X'$. These are precisely the contexts $D_{r,s}$ of
Theorem~\ref{thm:completion-contexts} in the main paper.

The canonical families do not, however, exhaust the full
orthogonality structure. For $n=3$, besides the nine canonical
contexts $C_{ij}$ of type $2A+2B$ and the nine contexts $D_{r,s}$
of type $A+B+X+X'$, there are six additional bases: three of type
$3A+X$ and three of type $3B+X'$. These are part of the exceptional
$F_4$ geometry described in Section~\ref{sec:exceptional} of the
main paper.

For $n=5$ and $n=7$, a different phenomenon occurs. There are
$2n^2$ orthogonal bases of type $2A+2B$, whereas the polar-secant
construction singles out only the $n^2$ canonical contexts
$C_{ij}$. Thus in both cases there is a second family of $n^2$
orthogonal bases involving only grid rays. This illustrates the
distinction, emphasized in the main paper, between the context
families canonically selected by the geometry and the full
orthogonality hypergraph of the underlying ray configuration.

For $n=5$ there is a further distinction between the orthogonality
structure of the original dual-grid configuration and its collinear
completion. In the $H_4$-compatible realization, the $50$ rays of
$Z_{5,5}$ are already orthogonally closed, even though the canonical
collinear completion adds ten further rays. This is verified
separately in Section~\ref{app:Z55-closure}.

\subsection*{Exact verification}

The SageMath code \texttt{completed\_orthogonality.sage} uses the explicit scalar-product
formulas rather than repeated symbolic simplification of the
coordinates. The required trigono\-metric values are represented as
exact algebraic numbers in \texttt{AA}/\texttt{QQbar}. The script
enumerates all orthogonal bases of the three realizations in the
table, verifies the canonical context families, and exhaustively
checks all admissible cyclic steps for $n=5$ and $n=7$. No
floating-point approximation or numerical tolerance is used.

\subsection*{Output}
{\footnotesize
\begin{verbatim}
	n = 3 a = 1 b = 1
	orthogonal bases: 24
	types: Counter({(2, 2, 0, 0): 9, (1, 1, 1, 1): 9, (0, 3, 0, 1): 3, (3, 0, 1, 0): 3})
	
	n = 5 a = 1 b = 2
	orthogonal bases: 75
	types: Counter({(2, 2, 0, 0): 50, (1, 1, 1, 1): 25})
	
	n = 7 a = 1 b = 2
	orthogonal bases: 147
	types: Counter({(2, 2, 0, 0): 98, (1, 1, 1, 1): 49})
	
	n = 5
	admissible pairs: 12
	possible numbers of bases: [50, 75]
	maximum: 75
	maximizing pairs: [(1, 2), (1, 3), (2, 1), (2, 4), (3, 1), (3, 4), (4, 2), (4, 3)]
	
	n = 7
	admissible pairs: 24
	possible numbers of bases: [98, 147]
	maximum: 147
	maximizing pairs: [(1, 2), (1, 5), (2, 1), (2, 3), (2, 4), (2, 6), (3, 2), (3, 5), (4, 2), 
                                	(4, 5), (5, 1), (5, 3), (5, 4), (5, 6), (6, 2), (6, 5)]
\end{verbatim}}
Thus, for these diagonal examples, the canonical projective
construction is unchanged while the full orthogonality hypergraph
varies with the Euclidean realization. In particular, the
$H_4$-compatible choice $(a,b)=(1,2)$ is a maximizing realization
for $n=5$, with $75=3\cdot5^2$ orthogonal bases.

This also shows that projectively equivalent realizations of the
same dual-grid incidence structure need not have the same full
orthogonality hypergraph. The projective construction determines the
canonical context families, whereas additional orthogonal bases may
appear for special Euclidean realizations.

\section{Orthogonal closure of the $Z_{5,5}$ realization}
\label{app:Z55-closure}

We record here the finite computation used in
Remark~\ref{rem:collinear-orthogonal-completion}. We consider
specifically the $H_4$-compatible realization of $Z_{5,5}$
determined by
\[
a=1,\qquad b=2,\qquad
\alpha=\frac{2\pi}{5},\qquad
\beta=\frac{4\pi}{5},\qquad
c=\phi^{-1},
\]
where $\phi=(1+\sqrt5)/2$.

The specification of the Euclidean realization is essential here.
As seen in Section~\ref{app:completed-orthogonality}, the full
orthogonality hypergraph of a dual-grid configuration may depend on
the admissible cyclic steps, even when the underlying projective
incidence construction is unchanged.

We say that a finite set of rays $Z\subset\PP^3(\RR)$ is
\emph{orthogonally closed} if, whenever three mutually orthogonal
rays of $Z$ are given, their common orthogonal complement also
belongs to $Z$. Since three mutually orthogonal rays in $\RR^4$
span a three-dimensional subspace, their common orthogonal
complement is a unique projective ray.

\begin{proposition}
	\label{prop:Z55-orthogonal-closure}
	The $H_4$-compatible realization of $Z_{5,5}$ with
	$(a,b)=(1,2)$ is orthogonally closed. More precisely, its $50$ rays
	contain $200$ mutually orthogonal triples, and each such triple has
	its unique orthogonal complement among the same $50$ rays. These
	triples form exactly $50$ orthogonal bases.
\end{proposition}

\begin{proof}
	The claim is verified by the SageMath code in \texttt{z55\_orthogonal\_closure.sage}. The
	script constructs the orthogonality graph directly from the
	scalar-product formulas of the dual-grid realization, enumerates all
	mutually orthogonal triples, and checks that each triple has exactly
	one common orthogonal ray among the original $50$ rays. It then
	identifies the resulting four-element sets and verifies that there
	are exactly $50$ distinct orthogonal bases.
\end{proof}

Thus the ten rays added in the passage
\[
Z_{5,5}\subset\widehat Z_5\simeq_{\mathrm{orth}} H_4^{\mathrm{proj}}
\]
are not forced by orthogonal closure. Their appearance is genuinely
a consequence of the collinear completion developed in
Section~\ref{sec:completion}. This provides a particularly sharp instance of the distinction between the projective completion and the full metric orthogonality structure.

Notice that the completed configuration $\widehat Z_5$ has $75$
orthogonal bases in this realization, as recorded in
Section~\ref{app:completed-orthogonality}. Thus the ten rays added
by collinear completion are not required to close the original
orthogonality structure, but once present they support $25$
additional orthogonal contexts.

\section{Exact verification of the $H_4$ identification}
\label{app:H4-verification}

We record here the explicit point correspondence and the exact
computation used in the proof of
Theorem~\ref{thm:H4-completion} of the main paper. Throughout this
section, equality of nonzero vectors is understood projectively. We use the numbering $P_1,\ldots,P_{60}$ of the projectivized
$H_4$ configuration adopted by Wi\'sniewska--Zi\k{e}ba with coordinates
coming from the realization of Harbourne--Migliore--Nagel--Teitler. 

The two
$25$-point grids may be displayed as
\[
G_1=
\begin{pmatrix}
	P_1   & P_{33} & P_{32} & P_{29} & P_{36}\\
	P_{31}& P_5    & P_{52} & P_{41} & P_{13}\\
	P_{34}& P_{58} & P_6    & P_{15} & P_{38}\\
	P_{35}& P_{37} & P_{14} & P_7    & P_{57}\\
	P_{30}& P_{16} & P_{42} & P_{51} & P_8
\end{pmatrix},
\qquad
G_2=
\begin{pmatrix}
	P_2   & P_{24} & P_{20} & P_{17} & P_{21}\\
	P_{19}& P_9    & P_{59} & P_{45} & P_{25}\\
	P_{23}& P_{55} & P_{10} & P_{27} & P_{43}\\
	P_{22}& P_{44} & P_{26} & P_{11} & P_{56}\\
	P_{18}& P_{28} & P_{46} & P_{60} & P_{12}
\end{pmatrix}.
\]
In each matrix, the five rows are the lines of one ruling of the
supporting quadric and the five columns are the lines of the other
ruling. Thus every row and every column consists of five collinear
points.

The remaining ten points are
\[
L_+=
\{P_3,P_{49},P_{50},P_{53},P_{54}\}
=\ell_{17}\cap H_4^{\mathrm{proj}},
\]
and
\[
L_-=
\{P_4,P_{39},P_{40},P_{47},P_{48}\}
=\ell_{24}\cap H_4^{\mathrm{proj}}.
\]

For the realization
\[
a=1,\qquad b=2,\qquad c=\phi^{-1},
\]
let $R$ be the orthogonal transformation displayed in
Theorem~\ref{thm:H4-completion}. The computation below verifies
exactly that
\[
R^{\mathsf T}R=I,
\qquad
R\bigl(\{A_{ij}\}\bigr)=G_1,
\qquad
R\bigl(\{B_{ij}\}\bigr)=G_2,
\]
and
\[
R(X_5)=L_+,
\qquad
R(X'_5)=L_-.
\]
Consequently,
\[
R(\widehat Z_5)=H_4^{\mathrm{proj}}.
\]

For completeness, the individual completion-ray correspondence is
\[
\begin{aligned}
	R(X_0)&=P_3,&
	R(X_1)&=P_{54},&
	R(X_2)&=P_{49},&
	R(X_3)&=P_{50},&
	R(X_4)&=P_{53},\\
	R(X'_0)&=P_4,&
	R(X'_1)&=P_{40},&
	R(X'_2)&=P_{47},&
	R(X'_3)&=P_{48},&
	R(X'_4)&=P_{39}.
\end{aligned}
\]

\subsection*{Exact verification}
The first part of the code \texttt{h4\_identification.sage} checks, over
$\mathbb Q(\sqrt5)$, the spectral data used in the proof of
Theorem~\ref{thm:H4-completion}: the duality of the two supporting
quadrics, their two eigenspaces, and the identification of
$\ell_{17}$ and $\ell_{24}$ with the corresponding spectral lines.
The second part passes to SageMath's field of algebraic numbers and
verifies the explicit orthogonal identification ray by ray.
Projective equality is tested by the vanishing of all
$2\times2$ minors of the corresponding vectors. No floating-point
approximation, numerical tolerance, or numerical recognition is used.

\subsection*{Output}
{\scriptsize\begin{verbatim}
	Verified spectral data:
	M2 = M1^(-1)
	ell_17 = P(E_{phi^(-1)})
	ell_24 = P(E_{-phi})
	
	FINAL EXACT VERIFICATION
	------------------------
	Verified: R^T R = I.
	Verified: R(A) = G1.
	Verified: R(B) = G2.
	Verified: R(X_5) = ell_17 cap H4^proj.
	Verified: R(X'_5) = ell_24 cap H4^proj.
	Verified: R(Zhat_5) = H4^proj.
\end{verbatim}}

\section{Generic projection of the maximizing $\widehat Z_7$ realization}
\label{app:Z77-geproci}

We record here the computation used in
Subsection~\ref{subsec:further-questions} of the main paper to test
the geproci property of the completed configuration $\widehat Z_7$.
We use the admissible realization
\[
a=1,\qquad b=2,
\]
which, as shown in
Section~\ref{app:completed-orthogonality}, attains the maximal number
$147=3\cdot7^2$ of orthogonal bases among the admissible cyclic
realizations considered there.

For this choice,
\[
\alpha=\frac{2\pi}{7},
\qquad
\beta=\frac{4\pi}{7},
\]
and
\[
c^2
=
-\frac{\cos(\alpha-\beta)}{\cos(\alpha+\beta)}
=
\frac{\cos(2\pi/7)}{\cos(\pi/7)}.
\]
The completed configuration consists of 
$49+49+7+7=112$ 
points.

If a reduced set of $112$ points in $\PP^2$ is a complete
intersection of curves of degrees $d\leq e$, then $de=112$.
Consequently,
\[
(d,e)\in
\{(1,112),(2,56),(4,28),(7,16),(8,14)\}.
\]
In particular, every such complete intersection is contained in a
nonzero curve of degree at most $8$. Thus, to disprove the geproci property, it is enough to exhibit one projection with distinct images for which $I_8=0$ and then
use upper semicontinuity.

\subsection*{Exact verification}

The SageMath code \texttt{z77\_geproci.sage} constructs the $112$ points of
$\widehat Z_7$, applies an explicit linear projection to $\PP^2$,
checks that the projected points remain distinct, and computes the
dimensions of the spaces $I_d$ for $7\leq d\leq14$. All
computations are performed in exact algebraic arithmetic.

\subsection*{Output}

The relevant part of the computation gives
\[
\begin{array}{c|c|c}
	d & \dim I_d
	& \dim I_d\text{ for a CI}(8,14)
	\\
	\hline
	7  & 0  & 0\\
	8  & 0  & 1\\
	9  & 0  & 3\\
	10 & 1  & 6\\
	11 & 3  & 10\\
	12 & 6  & 15\\
	13 & 10 & 21\\
	14 & 18 & 29
\end{array}
\]

The decisive discrepancy occurs already in degree $8$:
\[
\dim I_8=0.
\]
Hence the projected configuration is not contained in any nonzero
curve of degree at most $8$, and therefore cannot be a complete
intersection of $112$ reduced points.

Finally, this computation determines the behavior of a general
projection. As the center of projection varies in the open set for
which the $112$ images remain distinct, the dimension
 $h^0\!\left(
\mathcal I_{\pi_P(\widehat Z_7)}(8)
\right)$
is upper semicontinuous. Since it vanishes for the explicit
projection above, it vanishes on a nonempty open set of projection
centers. Therefore a general projection of this realization of
$\widehat Z_7$ is not a complete intersection, and hence
the maximizing realization of $\widehat Z_7$ is not geproci.

\end{document}